\documentclass[12pt]{article}
\usepackage{latexsym}
\usepackage{amssymb}
\usepackage{amsmath}
\usepackage{enumerate}
\usepackage{mathrsfs}
\usepackage{wasysym}
\usepackage{rawfonts}
\input{prepictex}
\input{pictex}
\input{postpictex}
\usepackage[OT2,OT1]{fontenc}
\def\cyr{%
\renewcommand\rmdefault{wncyr}%
\renewcommand\sfdefault{wncyss}%
\renewcommand\encodingdefault{OT2}%
\normalfont
\selectfont}
\numberwithin{equation}{section}
\DeclareMathAlphabet{\zap}{OT1}{pzc}{m}{it}
\DeclareTextFontCommand{\textcyr}{\cyr}
\def\be{\begin{equation}}
\def\ee{\end{equation}}
\def\bea{\begin{eqnarray*}}
\def\eea{\end{eqnarray*}}

\newcommand{\rad}{\text{\cyr   ya}}
\newcommand{\radius}{\mathsf{r}}

\def\CC{\mathbb C}
\def\RP{\mathbb{RP}}

\def\QQ{\mathbb Q}

\def\vol{\mbox{Vol}}

\newtheorem{main}{Theorem}
\DeclareMathOperator{\dft}{def}

\DeclareMathOperator{\ind}{ind}

\DeclareMathOperator{\Int}{Int }

\DeclareMathOperator{\sech}{sech }
\DeclareMathOperator{\tr}{tr}

\newtheorem{thm}{Theorem}[section]
\newtheorem{lem}{Lemma}[section]
\newtheorem{prop}{Proposition}[section]
\newtheorem{cor}[thm]{Corollary}

\newenvironment{proof}{\medskip \noindent
{\bf Proof.}}{\hfill \raisebox{-.2em}{\rule{.7em}{.8em}}
\\}

\def\ZZ{{\mathbb Z}}
\def\RR{{\mathbb R}}
\def\CP{{\mathbb C \mathbb P}}
\begin{document}

\title{Curvature, Cones, and Characteristic Numbers} 

\author{Michael Atiyah\thanks{Research supported in part by 
the Simons Center for Geometry and Physics.}  
   ~and Claude LeBrun\thanks{Research supported 
in part by  NSF grant DMS-0905159.}}

\date{}
\maketitle

 \begin{abstract}	
 We study Einstein metrics on  smooth compact $4$-manifolds
with an  edge-cone singularity of specified cone angle along an embedded
 $2$-manifold. To do so, we first  derive modified versions of the Gauss-Bonnet and signature
 theorems for arbitrary 
 Riemannian $4$-manifolds with edge-cone singularities, and then 
  show that these yield non-trivial obstructions in the 
Einstein case.
We then use these integral formul{\ae}  to 
 obtain interesting  information regarding 
 gravitational instantons which arise as limits of such edge-cone manifolds. 
\end{abstract}

\section{Introduction}

Recall \cite{bes} that a Riemannian manifold $(M,g)$ is said to be 
{\em Einstein}  if it has constant Ricci curvature; this is equivalent to 
requiring that the Ricci tensor $r$ of $g$ satisfy 
$$r=\lambda g$$
for some real number $\lambda$, called the Einstein constant of $g$. 
While one typically requires $g$ to be a smooth metric on $M$, 
it is sometimes interesting to consider generalizations where $g$
is allowed to have mild singularities. In the K\"ahler case, beautiful results \cite{brendedge,donedge,jmredge} 
have recently been 
obtained  regarding the situation in which   
 $g$ has specific  conical singularities along a submanifold of real codimension 
$2$. Einstein manifolds with such {\em edge-cone singularities}  are the main focus of 
the present article.

Let $M$ be a smooth $n$-manifold,  and let $\Sigma \subset M$ be a smoothly
embedded $(n-2)$-manifold. Near any point $p\in \Sigma$, we can thus
find local coordinates $(y^1,y^2, x^1, \ldots, x^{n-2})$ in which $\Sigma$
is given by $y^1=y^2=0$. Given any such adapted 
coordinate system, we then introduce an associated  
{\em transversal polar coordinate}  system  $(\rho,\theta, x^1, \ldots, x^{n-2})$
 by setting  $y^1=\rho \cos \theta$ and  $y^2= \rho \sin \theta$.
We define   a Riemannian edge-cone metric $g$ of cone angle
$2\pi\beta$ on $(M,\Sigma)$  to  be a smooth Riemannian metric 
 on $M-\Sigma$ which, for some $\varepsilon > 0$,  can be expressed as 
 \begin{equation}
\label{edgesum}
 g=\bar{g}+\rho^{1+\varepsilon}h
\end{equation}
 near any point of $\Sigma$, 
 where  the  symmetric tensor field $h$ on $M$ has infinite conormal regularity 
along $\Sigma$, and where 
 \begin{equation}
\label{edgedef}
\bar{g}= d\rho^2 + \beta^2 \rho^2 (d\theta+ u_jdx^j)^2
+ w_{jk} dx^jdx^k
\end{equation}
in  suitable transversal polar coordinate  systems; here  $w_{jk}(x) dx^jdx^k$
and  $u_j(x)dx^j$ are  a smooth metric and a smooth $1$-form on $\Sigma$.
(Our conormal regularity hypothesis means that 
  the components 
of $h$ in $(y, x)$ coordinates have infinitely many continuous derivatives with respect to 
$\partial/\partial x^j$, $\partial/\partial \theta$, 
and  $\rho\, \partial/\partial \rho$.)
Thus, an edge-cone metric $g$ behaves like
a smooth metric in directions parallel to $\Sigma$, but is modelled on 
a $2$-dimensional cone
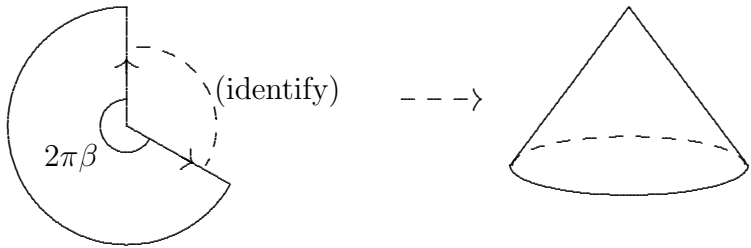
\begin{center}
\mbox{
\beginpicture
\setplotarea x from 0 to 270, y from 55 to 155
\circulararc 240 degrees from 30 145 center at 30 100 
\circulararc 240 degrees from 30 110 center at 30 100 
\ellipticalarc axes ratio 4:1 180 degrees from 175 85
center at 220 85
{\setlinear 
\plot  30 100 30 145 /
\plot  30 100 69 78  /
\plot  220 145 175 85   /
\plot  220 145 265 85   /
}
\put {$2\pi \beta$} [B1] at 5 85
\put {(identify)} [B1] at 83 110
\arrow <3pt> [1,3] from 26 120  to 26 125
\arrow <3pt> [1,3] from 47 88  to 51  86
\setdashes 
\arrow <3pt> [1,3] from 130 110  to 160 110
\circulararc -120 degrees from 30 130 center at 30 100
\ellipticalarc axes ratio 4:1 -180 degrees from 172 85
center at 216 85
\endpicture
}
\end{center}
in the transverse directions.

If an edge-cone  metric
on $(M,\Sigma)$ 
 is 
Einstein on $M-\Sigma$, we will call it  an Einstein edge-cone metric. 
For example, when $\beta < 1/3$, 
any 
K\"ahler-Einstein edge metric of cone angle $2\pi \beta$, in the sense of 
\cite{brendedge,donedge,jmredge}, can be shown 
   \cite[Proposition 4.3]{jmredge}  to be  an Einstein edge-cone metric. 
Another   
 interesting class, with $\beta=1/p$ for some  integer
$p \geq 2$,  is obtained by taking quotients of non-singular Einstein manifolds
by   cyclic groups of isometries for which the fixed-point set is purely of codimension $2$. 
Some explicit 
 Einstein edge-cone  metrics of much larger cone angle are also 
 known, as we will see   in \S \ref{zoo} below.  
Because of the wealth of  examples produced by such constructions, 
it   seems entirely reasonable to 
limit the present investigation  to   edge-cone metrics, as defined above. 
However, we do so purely 
as a matter of exigency. For example,  it    is remains  unknown whether there exist  
Einstein metrics with analogous  singularities for which the  cone angle varies along $\Sigma$. 
This is an issue  which  clearly merits thorough exploration. 

The Hitchin-Thorpe inequality \cite{bes,hit,tho} provides an important obstruction to the 
existence of Einstein metrics on $4$-manifolds. 
If $M$ is a smooth compact oriented $4$-manifold which admits 
a smooth  Einstein metric $g$, then the Euler characteristic $\chi$ and signature
$\tau$ of $M$ must satisfy the two inequalities 
$$
(2\chi \pm 3\tau )(M) \geq 0
$$
because both expressions are represented by Gauss-Bonnet-type integrals
where the integrands become non-negative in the Einstein case.
Note that this inequality hinges on several peculiar features of $4$-dimensional Riemannian geometry, and that 
no  analogous obstruction to the existence of Einstein metrics is currently  known in 
any dimension $\geq 4$.

In light of the current interest in Einstein metrics with edge-cone singularities, 
we believe it is interesting and natural to look for obstructions to their existence
which generalize our  understanding of the smooth case. Our main objective
here will be to prove the following version of the Hitchin-Thorpe inequality for edge-cone metrics: 

\begin{main}\label{oomph} 
 Let $(M,\Sigma)$ be a pair consisting of a smooth compact $4$-manifold and
a fixed smoothly embedded compact oriented surface. If $(M,\Sigma)$ admits
an Einstein edge-cone metric with cone angle $2\pi \beta$ along $\Sigma$,
then $(M,\Sigma )$ must satisfy the two  inequalities 
$$(2\chi \pm 3\tau) (M) \geq (1-\beta ) \left( 2\chi (\Sigma)  \pm (1+\beta ) [\Sigma]^2 \right)~.$$
\end{main}

To show this, we will   cast our net a good deal wider. In 
\S\S \ref{curvy}--\ref{argument},  we prove edge-cone generalizations 
of the $4$-dimensional Gauss-Bonnet and signature formul{\ae};  these results,
Theorems \ref{thing1} and \ref{thing2},  
 do not involve the Einstein condition, 
but rather apply to    arbitrary edge-cone
metrics on  compact $4$-manifolds.  We then zero in  on the Einstein case
in \S \ref{hti},  proving Theorem \ref{oomph} 
and  exploring some of  its implications. 
 But Theorems \ref{thing1} and \ref{thing2} have broader ramifications. 
 In \S  \ref{zoo}, 
we  apply  them to the   study of some explicit self-dual edge-cone  metrics, and explore 
the remarkable way  that  certain gravitational 
instantons arise  as limits of  edge-cone manifolds as  $\beta \to 0$. 

\pagebreak

\section{Curvature Integrals and Topology}
\label{curvy}

The Euler characteristic $\chi$ and signature $\tau$ of a smooth compact $4$-manifold $M$ 
 may both be calculated by choosing any smooth Riemannian metric $g$ on $M$,
 and then integrating appropriate universal quadratic polynomials
 in the curvature of $g$. When $g$ has an edge-cone singularity, however, 
 correction terms must be introduced in order to compensate for the singularity 
 of the metric along the given surface $\Sigma\subset M$. 
 
 \begin{thm}\label{thing1}
 Let $M$ be a smooth compact  oriented $4$-manifold, and 
 let $\Sigma \subset M$ be a smooth compact oriented embedded  surface. Then,
 for any edge-cone metric $g$ on $(M,\Sigma)$ with cone angle $2\pi \beta$, 
 \begin{equation}
  \label{prop1}
 \chi (M)  - (1-\beta) \chi (\Sigma ) = 
 \frac{1}{8\pi^2}\int_M \left(
 \frac{s^2}{24} + |W|^2 -\frac{|\mathring{r}|^2}{2}
 \right) d\mu ~.
\end{equation}
 \end{thm}
 
 \begin{thm}\label{thing2}
Let $M$ be a smooth compact  oriented $4$-manifold, and 
 let $\Sigma \subset M$ be a smooth compact oriented embedded  surface. 
  Then,
 for any edge-cone metric $g$ on $(M,\Sigma)$ with cone angle $2\pi \beta$, 
 \begin{equation}
  \label{prop2}
 \tau (M)  - \frac{1}{3} (1-\beta^2) [\Sigma ]^2 = 
 \frac{1}{12\pi^2}\int_M \left(
  |W_+|^2 - |W_-|^2
 \right) d\mu ~.
\end{equation}
 \end{thm}

 \noindent 
 Here $s$, $\mathring{r}$, and $W$ are the  scalar curvature, trace-free Ricci tensor, 
 and Weyl curvature of $g$, while  $W_+$ and $W_-$ are the self-dual and 
 anti-self-dual parts of $W$, 
 and  $d\mu$ is the metric 
 volume $4$-form. We follow standard conventions \cite{bes} by defining 
 $$|\mathring{r}|^2 := \mathring{r}_{jk}\mathring{r}^{jk}~, \hspace{.5in} |W|^2 := 
 \frac{1}{4} W_{jk\ell m}W^{jk\ell m}$$
 where the factor of $1/4$ arises from treating $W$ as  an element of 
 $\Lambda^2 \otimes \Lambda^2$; the point-wise norms of $W_\pm$ are defined 
 analogously, so that 
 $$|W|^2 = |W_+|^2+ |W_-|^2~.$$
 The expression $[\Sigma ]^2$ 
 denotes the self-intersection of the homology class of $\Sigma$ in 
 $H_2 (M)\cong H^2 (M)$, and coincides with the Euler class  of the normal bundle of
 $\Sigma$, paired with  the fundamental cycle of the surface. 
 
  We first discovered Theorems \ref{thing1} and \ref{thing2} 
 in the context of global-quotient orbifolds,  and we outline a method for  deducing the general case
from this special one  in \S \ref{argument}   below. 
 Another workable strategy, which we leave to the interested reader, 
   would   be  to apply the Gauss-Bonnet and
 signature theorems with boundary \cite{aps} to the complement of a tubular neighborhood 
 of $\Sigma\in M$, and then take  limits as the  radius of the tube tends to zero. 
However, we will instead begin here  by    giving a  complete and 
self-contained   proof by  yet a third
method,  of a purely differential-geometric  flavor.

For this purpose, observe that, by taking linear combinations,
    equations 
 (\ref{prop1}) and (\ref{prop2}) are equivalent to the pair of equations 
 \begin{eqnarray}
(2\chi + 3\tau) (M) -  \dft_+(\Sigma, \beta) 
&=&  \frac{1}{4\pi^2}\int_M \left(
 \frac{s^2}{24} + 2|W_+|^2  -\frac{|\mathring{r}|^2}{2}
 \right) d\mu ~~~\label{equiv1}\\
 (2\chi - 3\tau) (M) - \dft_-(\Sigma, \beta) 
&=&  \frac{1}{4\pi^2}\int_M \left(
 \frac{s^2}{24} + 2|W_-|^2  -\frac{|\mathring{r}|^2}{2}
 \right) d\mu ~~~\label{equiv2}
\end{eqnarray}
provided we define the {\em defects} $\dft_\pm(M,\Sigma, \beta)$ to be 
\begin{equation}
\label{defect}
\dft_\pm (\Sigma,\beta) = 2 (1-\beta)\chi(\Sigma) \pm (1-\beta^2 ) [\Sigma ]^2~.
\end{equation}
However, equations (\ref{equiv1})
and  (\ref{equiv2}) are interchanged  by simply reversing the orientation of $M$. 
 To prove Propositions  \ref{thing1} and \ref{thing2}, it  therefore suffices to 
 prove that (\ref{equiv1}) holds for any edge-cone metric on $(M,\Sigma)$ with cone angle $2\pi\beta$,
 with  defect $\dft_+(\Sigma, \beta)$ given by (\ref{defect}). 
 We now simplify the problem further by showing  that, for each 
 $(M,\Sigma)$ and  $\beta$,  it suffices to check 
 that   (\ref{equiv1}) holds   for a single  edge-cone metric  on $(M,\Sigma)$  of 
 cone angle $2\pi\beta$.

\begin{lem} 
\label{pram} 
Let $(M,\Sigma)$ be a smooth compact oriented $4$-manifold equipped with
a smooth compact oriented embedded surface, and let $\beta$ be a positive real number.
Fix two real constants 
  $a$ and $b$,  and  consider  the formula for $(a\chi + b\tau)(M)$ obtained by 
  taking the corresponding  linear combination of \eqref{prop1} and \eqref{prop2}. 
  If this formula  holds for one  edge-cone metric on $(M,\Sigma)$ of cone angle $2\pi \beta$, 
  it also holds for 
  every other  edge-cone metric   on $(M,\Sigma)$ 
of  the same  cone angle. 
\end{lem}
\begin{proof}
The Gauss-Bonnet and signature integrands are multiples of 
the $4$-forms
\begin{eqnarray*}
\Phi_{abcd} &=& (\star \mathcal{R})_{jk[ab}{\mathcal{R}^{jk}}_{cd]}
\\
\Psi_{abcd} &=& \mathcal{R}_{jk[ab}{\mathcal{R}^{jk}}_{cd]}
\end{eqnarray*}
corresponding to the Pfaffian and second symmetric polynomial of  curvature.
Given a
one-parameter family 
 $$g_{\zap t}:= g +{\zap t}\dot{g}+ O({\zap t}^2), $$ of Riemannian metrics, the 
 ${\zap t}$-derivative
 (at ${\zap t}=0$)  
 of the curvature operator ${\mathcal R}: \Lambda^2\to \Lambda^2$ is given by 
 $${{\dot{\mathcal{R}}^{ab}}_{cd}}=
-2\nabla_{[c}\nabla^{[a}\dot{g}^{b]}_{d]}+\dot{g}^{e[a}{\mathcal{R}^{b]}}_{ecd} ,$$
where $\nabla$ is the Levi-Civita connection of $g$, and where indices are raised and
lowered with respect to $g$. Thus
\begin{eqnarray*}
\dot{\Phi}_{abcd} &=& -2 \nabla_{[a|}\left({\epsilon}_{mnjk}{\mathcal{R}^{mn}}_{|cd}\nabla^{j}\dot{g}^k_{b]}\right) \\\ \dot{\Psi}_{abcd}&=& -4  \nabla_{[a|}\left( \mathcal{R}_{jk|cd}\nabla^{j}\dot{g}^k_{b]}
\right)
\end{eqnarray*}
where the cancellation of the purely algebraic terms is a nice exercise in the representation theory of $\mathbf{SO}(4)$. In other words, 
$$
\dot{\Phi} = d\phi  \qquad\qquad 
\dot{\Psi} = d\psi $$
for $3$-forms
\begin{eqnarray*}
\phi_{bcd} &=& - (\star \mathcal{R})_{jk[bc}\nabla^j\dot{g}^k_{d]}\\
\psi_{bcd} &=&-  \mathcal{R}_{jk[bc}\nabla^j\dot{g}^k_{d]}
\end{eqnarray*}
which obviously satisfy 
\begin{equation}
\label{notbad}
|\phi| , |\psi| \leq |\mathcal R| |\nabla \dot{g}|~.
\end{equation}

Now we have defined an edge-cone metric $g$ on $(M,\Sigma)$ of cone angle $2\pi \beta$ to be 
a tensor field which is smooth on $M-\Sigma$, and which can be written
as $\bar{g}+\rho^{1+\varepsilon}h$ for some $\varepsilon > 0$, where $h$
has infinite conormal regularity at $\Sigma$, and where the background metric $\bar{g}$ 
takes
the form 
$$
\bar{g}= d\rho^2 + \beta^2\rho^2 (d\theta +\nu )^2 + g_\Sigma
$$
in suitable transverse polar coordinates, where $\nu$ and $g_\Sigma$ are the pull-backs to a tubular neighborhood
of $\Sigma$ of a smooth $1$-form and a smooth Riemannian metric on $\Sigma$. 
While $\rho$ and $\theta$ can be  related to a system of smooth  coordinates  
$(y^1,y^2,x^1,x^2)$ by $y^1=\rho \cos \theta$, $y^2=\rho \sin \theta$, we now 
set $\tilde{\theta}= \beta\theta$, and 
introduce  new local coordinates on  transversely 
wedge-shaped regions by setting $\tilde{y}^1= \rho \cos \tilde{\theta}$, $\tilde{y}^2= 
\rho \sin \tilde{\theta}$. In these coordinates, $\bar{g}$ simply appears to 
be a smooth metric with $\mathbf{SO}(2)$-symmetry around $\Sigma$, while the tensor
field $h$ still has the same infinite conormal regularity as before. Thus the 
 first derivatives of the components of $g$ in $(\tilde{y}, x)$ coordinates 
 are smooth plus terms of order 
 $\rho^\varepsilon$, while the second derivatives  are no worse than 
 $\rho^{-1+\varepsilon}$. Since $g^{-1}$ is also continuous across $\Sigma$, it follows
 that  the  Christoffel symbols $\Gamma^{j}_{k\ell}$ of $g$ in  $(\tilde{y}, x)$ coordinates are bounded, 
 and that the norm $|\mathcal{R}|_g$ of the curvature tensor at worst blows up 
 like $\rho^{-1+\varepsilon}$.

 Given two edge-cone 
metrics   $g$ and $g^\prime$ on $(M,\Sigma)$ of the same cone angle $2\pi \beta$, 
we first apply a diffeomorphism to $(M,\Sigma)$ in order to arrange that the 
two choices of  radius functions $\rho$ and identifications 
of the normal bundle of $\Sigma$ with a tubular neighborhood
agree; thus, without interfering with our curvature integrals, we can assume that 
the two given choices of $\bar{g}$ differ only insofar as they involve different choices of 
$\nu$ and $g_\Sigma$. The $1$-parameter family $g_{{\zap t}} = (1-{\zap t}) g+ {\zap t}g^\prime$,
${\zap t}\in [0,1]$ is then a family of edge-cone metrics on $(M,\Sigma)$ 
of fixed cone angle $2\pi\beta$.

 We will now show that 
$\int_M\Phi_{g_{\zap t}}$ and $\int_M\Psi_{g_{\zap t}}$ are independent of ${\zap t}$. 
To see this, let us write 
$$
\frac{d}{dt} \Phi_{g_{\zap t}}= d \phi_{{\zap t}} \qquad \qquad
 \frac{d}{dt} \Psi_{g_{\zap t}}= d \psi_{{\zap t}}
$$
as above, and notice that  \eqref{notbad} tells us that 
$$ |\phi_{\zap t} | , |\psi_{\zap t} | \leq C \rho^{-1+\varepsilon}$$
for some positive constants $C$ and $\varepsilon$ determined by $g$ and $g^\prime$, 
since 
 the first derivatives of $\dot{g}= g^\prime-g$ 
 in $(\tilde{y},x)$ coordinates are bounded. 
Now let  $M_\delta$ denote the complement of a tube $\rho < \delta$ around
$\Sigma$. Then 
$$
\left| \frac{d}{dt} \int_{M_\delta} \Phi_{g_{\zap t}}\right| = \left| \int_{M_\delta} \frac{d}{dt}{\Phi}_{g_{\zap t}}\right|
= \left| \int_{\partial M_\delta}\phi_{g_{\zap t}}\right| < 
C \delta^{-1+\varepsilon} \vol^{(3)}(\partial M_\delta, g_{\zap t} )< 
\tilde{C} \delta^\varepsilon 
$$
for some ${\zap t}$-independent constant $\tilde{C}$. Integrating in ${\zap t}\in [0,1]$, we thus have 
$$
\left| 
 \int_{M_\delta} \Phi_{g^\prime} -  \int_{M_\delta} \Phi_{g}
 \right| < \tilde{C} \delta^{\varepsilon}~,
$$
and taking the limit $\delta \to 0$ therefore yields
$$
 \int_{M} \Phi_{g^\prime} = \int_{M} \Phi_{g}~.
$$
Replacing $\phi_{\zap t}$ with $\psi_{\zap t}$ similarly proves that 
 $\int_M\Psi_{g^\prime}=\int_M\Psi_g$.  Thus, if a given linear combination of 
 \eqref{prop1} and \eqref{prop2} is true for some edge-cone metric
 $g$, it is also true for any other edge-cone metric $g^\prime$ of the same cone angle.   
\end{proof}

We may thus focus our attention on proving \eqref{equiv1} for some  particular edge-cone metric
for  each $(M,\Sigma)$ and $\beta > 0$.   For any given metric, 
let us therefore  adopt the provisional  notation 
\begin{equation}
\label{strat1a}
\Upsilon = \left(\frac{s^2}{24} + 2|W_+|^2  -\frac{|\mathring{r}|^2}{2}
 \right) d\mu
 \end{equation}
for the  $4$-form appearing as the integrand in (\ref{equiv1}). 
If $g$ is an edge-cone metric of cone angle $2\pi\beta$ on $(M,\Sigma)$ and if
$g_0$ is a smooth metric on $M$, then (\ref{equiv1}) is equivalent  to the claim
that 
 \begin{equation}
\label{strat1b}
\int_M \left( \Upsilon_g - \Upsilon_{g_0}
\right) = -4\pi^2 \dft_+(\Sigma, \beta )~,
\end{equation}
since integral of $\Upsilon_{g_{0}}$ is $4\pi^2 (2\chi + 3\tau )(M)$ 
by the 
standard Gauss-Bonnet and signature theorems. 
Of course, if we can actually arrange for 
$g$ and $g_0$  to exactly agree on the complement of some 
tubular neighborhood $\mathcal V$ of $\Sigma$, this reduces to the statement that 
$$
\int_{{\mathcal V}-\Sigma} \left( \Upsilon_g - \Upsilon_{g_0}
\right) = -4\pi^2 \dft_+(\Sigma, \beta )
$$
since $\Upsilon_g-\Upsilon_{g_0}$ is then supported in ${\mathcal V}$, and 
$\Sigma$ has $4$-dimensional measure zero. 
In the  proof that follows, we  will  not only choose $g$ and $g_0$ to be  related
in this manner, but  also   arrange for   both of them to be K\"ahler on $\mathcal V$. Now 
$$|W_+|^2 = \frac{s^2}{24}$$
for any K\"ahler metric on a compatibly oriented 
$4$-manifold, so (\ref{strat1a})   simplifies in the K\"ahler case to become 
\begin{equation}
\label{strat2}
\Upsilon = \frac{1}{2} \left(\frac{s^2}{4} - |\mathring{r}|^2
 \right) d\mu 
 = \varrho \wedge \varrho
 \end{equation}
 where $\varrho$ is the Ricci form. This will reduce the  problem  to showing that
\begin{equation}
\label{strat3}
\int_{{\mathcal V}-\Sigma}\left(\varrho_g\wedge \varrho_g - \varrho_{g_0}\wedge  \varrho_{g_0}\right)
= -4\pi^2 (1-\beta)\left[ 2 \chi(\Sigma) + (1+\beta ) [\Sigma ]^2\right]
\end{equation}
and, by exploiting some 
key  properties of the Ricci form, this will follow by an application of Stokes' theorem. 

Our choice of  $g_0$ will involve   an auxiliary function 
$F(t)$, which we will now  construct. First let $f: \RR \to  \RR^+$ be a smooth 
positive function  with 
\begin{equation}
\label{grecian}
\int_0^1 f(t)dt = \frac{1}{\beta}
\end{equation}
such that 
$$
f(t)=
\begin{cases}
1/\beta& \mbox{ when } t \leq \frac{1}{2} \mbox{ and }\\
t^{\beta -1}& \mbox{ when } t \geq 1.
\end{cases}
$$
We next define $F(t)$ up to a  constant of integration by setting 
\begin{equation}
\label{urn}
\frac{dF}{dt}= \frac{1}{t}\int_0^t f(x) dx ~,
\end{equation}
where this definition of course entails that $F^\prime \equiv 1/\beta$ for $t\leq 1/2$. 
Because $F$ consequently solves the differential equation 
 \begin{equation}
\label{ode}
\frac{d}{dt}\left( t ~\frac{dF}{dt}\right) = f(t) 
\end{equation}
it follows that the smooth, rotationally symmetric metric on the $\zeta$-plane $\CC$
 with K\"ahler form  
\begin{equation}
\label{shape}
i\partial\bar{\partial}F (|\zeta |^2) = i f(|\zeta |^2)~ d\zeta\wedge \bar{\zeta}
\end{equation}
is Euclidean near the origin and coincides with a standard cone of perimeter angle $2\pi \beta$
outside the unit disk. Condition (\ref{grecian}) guarantees that $F^\prime (1) = 1/\beta$,
so inspection of the differential equation  (\ref{ode}) tells us  that 
$$F(t)= \frac{t^\beta}{\beta^2}+  B$$
for all $t\geq 1$,  for a constant  of integration $B$ which we may take to vanish.

 \begin{prop} 
 \label{stroller}
 For any  compact oriented pair $(M^4 ,\Sigma^2)$   and any positive real number 
 $\beta$, there is an edge-cone metric $g$ on  $(M,\Sigma)$ with cone angle $2\pi \beta$
 such that {\rm (\ref{equiv1})} holds, with defect $\dft_+$  given by {\rm (\ref{defect})}. 
 \end{prop}
 
\begin{proof}
 Choose any 
metric $g_\Sigma$ on $\Sigma$,  and 
remember that, in conjunction  with the given orientation, its conformal class
 makes $\Sigma$  into a 
compact complex curve; in particular,  the area form $\alpha$ of $g_\Sigma$ then becomes 
its  K\"ahler form.   Let $\varpi : E\to \Sigma$ be the normal bundle of 
$\Sigma\subset M$, and choose an inner product $\langle ~, ~\rangle$ on $E$;
this reduces the structure group of $E$ to $\mathbf{SO}(2)=\mathbf{U} (1)$, and so makes
it into a complex line bundle. Next, we choose 
 a  $\langle ~, ~\rangle$-compatible connection $\nabla$ on $E$  whose
  curvature  is a constant multiple of $\alpha$ on each connected component of $\Sigma$. Viewing $\nabla^{0,1}$ as a $\bar{\partial}$-operator on $E$ then makes
 it into a holomorphic line bundle over $\Sigma$, in a unique manner that 
identifies  $\nabla$ with the Chern connection induced by $\langle ~, ~\rangle$ and the 
holomorphic  structure. Thus 
    the curvature of $\nabla$ 
is   $- i \kappa  \alpha$, where the locally constant real-valued function $\kappa$
takes the value 
$$\kappa|_{\Sigma_j}  = \frac{2\pi\int_{\Sigma_j} c_1(E)}{\int_{\Sigma_j} \alpha} 
=\frac{2\pi ~[\Sigma_j ]^2}{\int_{\Sigma_j} \alpha}$$ 
on the $j^{\rm th}$ connected component 
$\Sigma_j$ of $\Sigma$. 

We will let $t: E\to \RR$ denote the square-norm function $t( v) = \| v \|^2$, and 
our computations will involve various closed $(1,1)$-forms expressed as  
$i \partial \bar{\partial} u(t)$ for various  functions $u(t)$. 
To understand such expressions explicitly, first choose
a local coordinate $z$ on $\Sigma$ so that near the origin 
$$\alpha = i [1 + O(|z|^2)] dz\wedge d \bar{z}, $$
and then choose a local trivialization 
of $E$ determined by a local holomorphic section $\xi$ with vanishing covariant derivative and  unit norm at the origin. Then the function ${\zap h} :=  \|\xi\|^2$ satisfies 
$$
{\zap h} = 1 + O(|z|^2), \hspace{.5in} \partial \bar{\partial}  {\zap h} = -\kappa ~dz\wedge d\bar{z} + O(|z|),
$$
because   $-\partial \bar{\partial} \log {\zap h}= - i \kappa  \alpha$ is the curvature 
 of $E$. Thus, introducing a fiber coordinate $\zeta$ associated with the local trivialization, 
then, near the $\zeta$-axis which represents the fiber over $z=0$, we have 
$$
i \partial \bar{\partial} u(t)= i (tu^\prime)^\prime d\zeta\wedge  d\bar{\zeta} - i\kappa  (t u^\prime) dz \wedge d\bar{z} +  O(|z|)
$$
where $t=|\zeta|^2$ along the $\zeta$-axis. Since the chosen point $z=0$ of $\Sigma$ was in fact
arbitrary,   this calculation of course  actually
computes $i \partial \bar{\partial} u(t)$  along {\em any} fiber.  
For example,  consider the $(1,1)$-forms 
\begin{equation}
\label{conehead}
\omega = \lambda \varpi^*\alpha + i \partial\bar{\partial} \left(\beta^{-2} t^\beta \right)
\end{equation}
and
\begin{equation}
\label{flathead}
\omega_0 = \lambda \varpi^*\alpha + i \partial\bar{\partial} F(t)
\end{equation}
on $E$, for some large positive constant $\lambda$, where $F$ is given by \eqref{grecian}. Explicitly, these are given 
along the $\zeta$-axis of our coordinate system by 
\begin{eqnarray*}
\omega~ &=&  i t^{\beta-1} d\zeta\wedge  d\bar{\zeta} + i (\lambda - \frac{\kappa }{\beta}t^\beta)  dz \wedge d\bar{z} +O(|z|) \\
\omega_0 &=&   i  f(t)  d\zeta\wedge  d\bar{\zeta}+ i (\lambda - \kappa  t F^{\prime}(t))  dz \wedge d\bar{z} +O(|z|) 
\end{eqnarray*}
so, for  $\lambda$  sufficiently large, 
these  are the  K\"ahler forms of  K\"ahler metrics 
$\tilde{g}$ and $\tilde{g_0}$ defined on, say,  the region $0<t<3$. Notice 
that $\tilde{g}_0$ is smooth across $t=0$, and that we have arranged for $\tilde{g}$ and 
$\tilde{g}_0$ to coincide when  $t> 1$. Also  observe that $\tilde{g}$ becomes
 a genuine edge-cone metric on $(M, \Sigma)$ after making the coordinate change $\zeta = (\beta \rho)^{1/\beta} e^{i \theta}$. While one could  
 object that this coordinate change actually represents
 a self-homeomorphism of $M$ which is only smooth away from  $\Sigma$, 
  this is completely harmless  for present purposes, since  the relevant
 curvature integrals will  actually be performed on $M-\Sigma$.

We now identify $E$ with a tubular neighborhood
$\mathcal U$ of $\Sigma$ via some diffeomorphism, and, for
any real number $T> 0$, we let   ${\mathcal U}_T\subset {\mathcal U}$ denote 
the closed tubular neighborhood 
corresponding to the region  $t\leq T$ of $E$. We then use a cut-off function to extend $\tilde{g}$ and $\tilde{g}_0$ to  $M$ 
as Riemannian metrics,  
in such a way that they exactly agree on the complement of ${\mathcal U}_1$,
but are undamaged by the cut-off on  ${\mathcal U}_2$. Calling these extensions
$g$ and $g_0$, respectively, we then see that   our basic desiderata have all been fulfilled: 
$g$ is an edge-cone metric of cone angle $2\pi \beta$ on $(M,\Sigma)$, $g_0$ is a smooth
Riemannian metric on $M$, both are K\"ahler on a tubular neighborhood 
${\mathcal V} = \Int  {\mathcal U}_2$ of $\Sigma$, and the two metrics agree on the complement of
a smaller  tubular neighborhood ${\mathcal U}_1$ of $\Sigma$. 

Since $i$ times the Ricci form is the curvature of the canonical line bundle, 
\begin{equation}
\label{ricky}
\varrho-\varrho_0 = d\varphi
\end{equation}
where the $1$-form 
\begin{equation}
\label{fatso}
\varphi = i \partial \log (V/V_0) 
\end{equation}
is defined in terms of  the ratio $V/V_0= d\mu_g/d\mu_{g_0}$  of the volume forms of $g$ and $g_0$. 
Thus
$$
\varrho^2 - \varrho_0^2 = (\varrho- \varrho_0) \wedge  (\varrho + \varrho_0)=  d \left(
 \varphi \wedge [ 2\varrho_0 + d\varphi ] 
\right) 
$$
and Stokes' theorem therefore tells us that  
\begin{equation}
\label{stokes}
\int_{M-{\mathcal U}_\epsilon}\left( \Upsilon_g -\Upsilon_{g_0}\right) = 
 \int_{{\mathcal U}_2-{\mathcal U}_\epsilon}
 d \left(
 \varphi \wedge [ 2\varrho_0 + d\varphi ] 
\right) 
 = - \int_{S_\epsilon}  \varphi \wedge [ 2\varrho_0 + d\varphi ] 
\end{equation}
where   the level set $S_\epsilon$ defined by $t= \epsilon$ has been given 
 the outward pointing orientation
relative to $\Sigma$. 
However, relative to 
 the basis provided by  the $4$-form  $- dz\wedge d\bar{z}\wedge d\zeta\wedge d \bar{\zeta}$ 
along the $\zeta$-axis, the volume forms of $g$ and $g_0$ are 
represented by the component functions  
$$
V= t^{\beta -1} (\lambda - \frac{\kappa }{\beta}t^{\beta}) ~~\mbox{ and }~~
V_0=  (\lambda - \kappa   t F^\prime (t))f(t)   
$$
where the latter simplifies when $t < \frac{1}{2}$ to become 
$$
V_0= 
(\lambda - \frac{\kappa }{\beta} t )\beta^{-1} ~. 
$$
Hence the $1$-form defined by (\ref{fatso}) is given by 
$$
\varphi = i \left(( \beta -1) - \frac{\kappa  t^{\beta}}{\lambda - \frac{\kappa }{\beta}t^\beta} +  \frac{\kappa t/\beta }{\lambda - \frac{\kappa }{\beta}t}\right) \partial \log t 
$$
in the region ${\mathcal U}_{1/2}-\Sigma$. Restricting this to $S_\epsilon$, then along
the $\zeta$-axis this expression is just 
\begin{eqnarray}
\varphi
&=&  \left( (1-\beta) +\frac{\kappa  \epsilon^{\beta}}{\lambda - \frac{\kappa \epsilon}{\beta}\epsilon^{\beta}} -  \frac{\kappa  \epsilon /\beta }{\lambda - \frac{\kappa }{\beta}\epsilon}\right)   ~d\theta \nonumber 
\\&=&  \left[ (1- \beta ) + O (\epsilon^{\min (1, \beta)} ) \right]  d\theta \label{around}
\end{eqnarray}
since  $\zeta= \sqrt{\epsilon} e^{i\theta}$ along the intersection of  $S_\epsilon$
and the $\zeta$-axis. On the other hand, 
\begin{eqnarray*}
d\varphi &=& -i\partial\bar{\partial} \log (V/V_0)\\
 &= & i\kappa    \left( (\beta -1) - \frac{\kappa  t^{\beta}}{\lambda - \frac{\kappa }{\beta}t^\beta} + \frac{\kappa  t/\beta }{\lambda - \frac{\kappa }{\beta}t}\right) dz\wedge d\bar{z} + U(t) d\zeta \wedge d\bar{\zeta}
\end{eqnarray*}
along the $\zeta$-axis, for some function $U(t)$. Since this calculation is valid along any fiber, 
it follows that, for $0< \epsilon  <1$, 
$$j_\epsilon^* d\varphi = \left[  (\beta -1) + O (\epsilon^{\min (1, \beta)} )\right]   j_\epsilon^*\varpi^* (\kappa \alpha)$$ 
where   $j_\epsilon : S_\epsilon \hookrightarrow M$
denotes the inclusion map, and $\varpi : E\to \Sigma$ once again denotes the bundle projection.  
Similarly, letting $j : \Sigma \hookrightarrow M$ be the inclusion,  we have 
$$j_\epsilon^* \varrho_0  = j_\epsilon^*\varpi^* j^* \varrho_0 + O(\epsilon)$$
since $\varrho_0$ is smooth across $\Sigma$ and is invariant under the action of 
$S^1=\mathbf{U} (1)$ on $E$.  Integration over the fibers of $S_\epsilon \to \Sigma$ therefore yields
\begin{equation}
\label{victory}
\int_{S_\epsilon} \varphi \wedge [ 2\varrho_0 + d\varphi ] = 
2\pi  (1-\beta)  \int_\Sigma [2\varrho_0 +   (\beta -1) \kappa\alpha ] ~~~+ O (\epsilon^{2\min(1,\beta)})
\end{equation}
by virtue of  (\ref{around}).
But since $\varrho_0/2\pi$  represents $c_1(T^{1,0}E)=  c_1(T^{1,0}\Sigma)+ c_1(E)$ in deRham cohomology, and since $c_1(E)$ is similarly represented by $\kappa  \alpha/2\pi$, we have 
\begin{eqnarray*}
 \int_\Sigma [2\varrho_0 +   (\beta -1) \kappa\alpha ] 
 &=& 2\pi \left[2 \left(\mathbf{c_1}(T^{1,0}\Sigma )
 +\mathbf{c_1}(E) \right)+ (\beta -1)\mathbf{c_1}(E) \right] \\
 &=& 2\pi \left[2{\mathbf c_1}(T^{1,0}\Sigma) +  (\beta +1){\mathbf c_1}(E) \right] \\
  &=& 2\pi \left(2\chi(\Sigma) +  (\beta +1)[\Sigma ]^2 \right)
\end{eqnarray*}
where the boldface Chern classes have been evaluated  on the homology class of $\Sigma$. 
Plugging this into (\ref{victory}), we obtain  
$$
\int_{S_\epsilon} \varphi \wedge \left[ 2\varrho_0 + d\varphi \right] = 
4\pi^2 (1-\beta)  \left(2\chi(\Sigma) +(1+\beta ) [\Sigma ]^2 \right)
~~~+ O (\epsilon^{2\min(1,\beta)})~.
$$
Substituting this   into (\ref{stokes}) and  taking the limit as $\epsilon \to 0$ thus yields  
\begin{equation}
\label{lap}
\int_{M}\left( \Upsilon_g -\Upsilon_{g_0}\right) = -4\pi^2  (1-\beta)  \left[2\chi(\Sigma) +  (1+ \beta  )[\Sigma ]^2\right]
\end{equation}
which is exactly the sought-after identity (\ref{strat1b}) for the particular metrics $g$ and $g_0$.
Applying the Gauss-Bonnet and signature fomul{\ae} to the smooth metric $g_0$
now transforms 
(\ref{lap}) into  
$$
\frac{1}{4\pi^2} \int_{M} \Upsilon_g = (2\chi + 3\tau ) (M) -  (1-\beta)  \left[2\chi(\Sigma) +  (1+ \beta  )[\Sigma ]^2\right] 
$$
which is exactly (\ref{equiv1}), with defect $\dft_+(M, \Sigma )$ given by (\ref{defect}).
 \end{proof}

Theorems \ref{thing1} and \ref{thing2} now follow. 
Indeed,   given any $(M,\Sigma)$ and $\beta$, 
Proposition \ref{stroller}   shows that \eqref{equiv1} holds for some edge-cone
metric $g$ on $(M, \Sigma)$ of cone angle $2\pi\beta$, and  Lemma \ref{pram} thus shows
that the same is true of any edge-cone metric on any $(M,\Sigma)$ for any $\beta$. 
Applying this conclusion to the 
orientation-reversed manifold $\overline{M}$  shows  that \eqref{equiv2} also holds, and
taking appropriate linear combinations then proves 
Theorems  \ref{thing1} and \ref{thing2}.

So far, we have assumed  that $\Sigma$ is oriented, but this is not  essential. 
Of course,  \eqref{prop1}  makes perfectly good sense even 
if $\Sigma$ is non-orientable,  but more must be said about 
\eqref{prop2}. 
As long as $M$ is oriented, the Euler number of the normal bundle of $\Sigma\subset M$
will have the  twisted coefficients needed to be consistently integrated on $\Sigma$,
and  this normal bundle will therefore have
 a well-defined  Euler number which  counts the zeroes, with multiplicities, 
 of a  generic section. When $\Sigma\subset M$ is non-orientable,  we now decree that 
 $[\Sigma]^2$ is to be interpreted in 
 \eqref{prop2} as   meaning
 the Euler number of its normal bundle 
 rather than being defined in terms of homology classes on $M$. 
 Now this Euler number can  be calculated by passing to an oriented  double cover of 
 $\Sigma$, integrating the Euler class of the pull-back, and then dividing by $2$;
 meanwhile,  the  correction term
  in \eqref{prop2} is represented by an integral   supported in a 
tubular neighborhood of $\Sigma$, so one  can similarly compute it by 
 passing to  a double cover of a tubular neighborhood of 
$\Sigma$ and then dividing by $2$. This covering  trick   allows us to  prove 
Theorem \ref{thing2} even when $\Sigma$ is non-orientable, 
and Theorem  \ref{thing1}   even if  neither $M$ nor $\Sigma$  is orientable. 


This observation  has a useful  corollary. 
Suppose that $M$ admits an  almost-complex structure 
$J$ and that $\Sigma \subset M$ is totally real with respect to $J$,
in the sense that $T\Sigma \cap J (T\Sigma) =0$ at every point of $\Sigma$. 
We now give $M$ the orientation induced by $J$, but  emphasize  that 
$\Sigma$ might not even  be  orientable. If $e_1,e_2$ is a basis $T\Sigma$ at some point, 
we now  observe that $(e_1,e_2, Je_1, Je_2)$ is then a {\em reverse-oriented}
basis for $TM$.  Under these circumstances, 
the Euler number of the normal bundle of $\Sigma$,
in the sense  discussed above, therefore equals $-\chi (\Sigma )$. In light of our previous
remarks, we thus obtain a close cousin  of Theorem \ref{thing2}:

\begin{prop} \label{thing3}
Let $(M,J)$ be an almost-complex $4$-manifold, equipped with 
the orientation induced by $J$, 
and let 
$\Sigma\subset M$ be a (possibly non-orientable) surface which 
is totally real with respect to $J$. Then, for any edge-cone metric  $g$ 
  of cone angle $2\pi \beta$ on $(M,\Sigma)$, 
\begin{equation}
\label{prop3}
 \tau (M)  + \frac{1}{3} (1-\beta^2) \chi (\Sigma ) = 
 \frac{1}{12\pi^2}\int_M \left(
  |W_+|^2 - |W_-|^2
 \right) d\mu ~.
\end{equation}
\end{prop}

\pagebreak 
 
\section{Indices and Orbifolds}
 \label{argument}

 While the calculations used in  \S \ref{curvy} suffice to 
  prove Theorems \ref{thing1} and \ref{thing2}, 
  they hardly provide a transparent explanation of
   the detailed structure of   equations  
  (\ref{prop1}) and (\ref{prop2}). 
In this section, we will describe another method for obtaining 
these formul{\ae} that makes the edge-cone corrections seem
a great deal less mysterious. For brevity and clarity, we will 
confine ourselves  to providing  a second proof of Theorem 
 \ref{thing2}.  The same method can also be used to prove Theorem 
 \ref{thing1}, but several more elementary proofs are also possible in this case.

Our approach is  based on 
 the $\mathbf{G}$-index theorem \cite{indexiii}, so we  begin by recalling  
 what this tells us about  the signature operator in   four dimensions. 
    If a  finite group $\mathbf{G}$ 
 acts on a compact  oriented connected $4$-manifold $X$,  we can choose a
 $\mathbf{G}$-invariant  decomposition 
 $$H^2 (M, \RR) = H^+ \oplus H^-$$
of the second cohomology  into subspaces on which the intersection form 
is positive- and negative-definite; for example, we could equip $X$ with 
a   $\mathbf{G}$-invariant Riemannian metric $\hat{g}$, and let $H^\pm$ consist of those 
de Rham classes which 
self-dual or anti-self-dual harmonic representatives, respectively, with respect to this metric. 
For each  ${\zap g}\in \mathbf{G}$, 
we then let ${\zap g}_*$ denote the induced action of ${\zap g}$ on $H^2(X, \RR)$, and
set 
$$\tau ({\zap g}, X) = \tr ( {\zap g}_* |_{H^+}) -  \tr ( {\zap g}_* |_{H^-}) ~.$$
In particular,   $\tau ( 1, X)$ coincides with the usual signature $\tau (X)$. 
By contrast,  when ${\zap g}\neq 1$,     $\tau ( {\zap g}, X)$
is instead expressible in terms  of the fixed-point set $X^{\zap g}$ of ${\zap g}$. 
To do this, we first express   $X^{\zap g}$ as a disjoint union of isolated fixed points $x_j$ 
and compact  surfaces $\hat{\Sigma}_k$; at each $\hat{\Sigma}_k$,
${\zap g}$ then 
acts by rotating the normal bundle through some angle $\vartheta_k$, whereas
at each isolated fixed point $x_j$, ${\zap g}$ acts on $T_{x_j} X$ by rotating 
through  angles $\alpha_j$ and $\beta_j$ in a pair of orthogonal $2$-planes. 
 With these conventions, the $4$-dimensional case of  the relevant fixed-point formula  \cite[Proposition (6.12)]{indexiii}
 becomes 
 \begin{equation}
\label{attwo}
\tau({\zap g},X)=-\sum_j \cot\frac{\alpha_j}{2}\cot \frac{\beta_j}{2}+ \sum_k \left(\csc^2 \frac{\theta_k}{2}\right) [\hat{\Sigma}_k]^2
\end{equation}
where $j$ and $k$ respectively run over the $0$- and $2$-dimensional components 
of the fixed-point set $X^{\zap g}$. 

Let $M$ denote  the orbifold $X/\mathbf{G}$, and 
notice that 
 the $\mathbf{G}$-invariant
subspace $H^2(X, \RR)^{\mathbf{G}}$
of $H^2(X,\RR)$ can be   identified with 
$H^2(M, \RR)$ 
  via pull-back. Since 
$$\frac{1}{|\mathbf{G}|}\sum_{{\zap g}\in \mathbf{G}} {\zap g}_*
: H^2(X,\RR)\to H^2(X, \RR)^{\mathbf{G}} 
$$
is the $\mathbf{G}$-invariant projection,
and since the cup product commutes with pull-backs,  
we therefore have 
\begin{equation}
\label{atthree}
\tau (M) = \frac{1}{|\mathbf{G}|}  \left[ \tau(X) + \sum_{{\zap g}\neq 1} \tau({\zap g},X) \right].
\end{equation}

We now specialize our discussion by assuming
that the action has a fixed point, but that no  fixed point is isolated. 
Thus, if $\mathbf{G}\neq \{ 1\}$, 
there must be at least one 
fixed surface $\hat{\Sigma}_k$, and the induced action on the normal bundle of each 
such $\hat{\Sigma}_k$ must be 
effective. Hence 
 $\mathbf{G}=\ZZ_p$  for some positive integer $p$.  Moreover,
  for each $k$,   the exponentials $e^{i\vartheta_k}$
of  the rotation angles $\vartheta_k$ appearing  in (\ref{attwo}) 
must sweep through all the $p^{\rm th}$ roots of unity as 
${\zap g}$ runs through  $\ZZ_p$. 
 We can therefore   rewrite (\ref{atthree}) as 
\begin{eqnarray}
\label{resto}
\tau (M)  = \frac{1}{p}\left[ \tau (X)  +  \left(\sum_{k=1}^{p-1}\csc^2 \left[ \frac{k\pi}{p}\right] \right) [\hat{\Sigma}]^2\right]
\end{eqnarray}
where $\hat{\Sigma}=X^{\mathbf{G}}=\cup_k\hat{\Sigma}_k$. 
On the other hand, as pointed out by   Hirzebruch   \cite[\S 4]{hirzrr}, the
 trigonometric sum in \eqref{resto} has an algebraic simplification 
\begin{equation}
\label{reside}
\sum_{k=1}^{p-1}\csc^2 \left( \frac{k\pi}{p} \right)= \frac{p^2-1}{3}~,
\end{equation}
as  can be proved using the 
 the Cauchy residue theorem. 
 
Now since $X^{\ZZ_p}= \hat{\Sigma}$ has  been assumed to be purely of codimension $2$, 
$M=X/\ZZ_p$ is a manifold, and comes equipped with a surface
$\Sigma\subset M$ which is  the image of $\hat{\Sigma}$. Observe,
moreover,   that $[\Sigma]^2  = p [\hat{\Sigma}]^2$. 
Substituting (\ref{reside}) into (\ref{resto}) therefore yields 
\begin{equation}
\label{rest2}
\tau (M)  -\frac{1}{3} (1-p^{-2}) [\Sigma ]^2=  \frac{1}{p}\tau (X)~.
\end{equation}
However, the usual signature theorem tells us that 
$\tau (X) = \frac{1}{3} \int_X p_1 (TX)$, and this allows us
to rewrite  the right-hand side of (\ref{rest2})  as 
$$
\frac{1}{p}\cdot \frac{1}{12\pi^2}\int_X \left(
  |W_+|^2 - |W_-|^2
 \right)_{\hat{g}} d\mu_{\hat{g}} 
 = \frac{1}{12\pi^2}\int_M \left(
  |W_+|^2 - |W_-|^2
 \right)_g d\mu_g
$$
because $(X-\hat{\Sigma},\hat{g})$ is a $p$-sheeted cover of $(M-\Sigma, g)$. 
Since $g$ is an edge-cone metric on $(M,\Sigma)$ with $\beta = 1/p$, 
we have therefore obtained a quite different proof of  (\ref{prop2}) in this special case. 

The global quotients  with $\beta = 1/p$ we have just analyzed  
 constitute  a special class of 
orbifolds. Without recourse to the results in the previous section,
one can similarly show that Theorem \ref{thing2} also applies 
to arbitrary  orbifolds with singular set of pure codimension $2$ and cone angle $2\pi/p$, 
 even when the spaces in question are not  global quotients. 
 As is explained  in 
 Appendix \ref{apps}, the 
  best way  of proving this 
involves the  Index Theorem for transversely elliptic operators, and 
also yields results in higher dimensions. 
We will now   assume  this more general fact,and see how it leads to a 
different  proof of Theorem \ref{thing2}.  

To do so, we  revisit  the global quotients discussed above, but now turn the 
picture upside down by  letting  
$M$ play the role previously assigned to $X$. That is, we now assume that 
there is an effective action of $\ZZ_q$ on $M$ with fixed point set 
$\Sigma$, and set $Y= M/\ZZ_q$. Let $\varpi : M\to Y$
be the quotient map, and let 
$\check{\Sigma} = \varpi (\Sigma )$.  
 Chose an  orbifold metric  $\check{g}$ 
  of cone angle $2\pi /p$ on $(Y,\check{\Sigma})$, and assume, per the 
  above discussion,  that 
 \eqref{prop2} is  already 
  known to hold for orbifolds. Our previous argument tells us that 
 \begin{equation*}
\tau (Y)  =\frac{1}{q}\tau (M)  + \frac{1}{3}  (1-q^{-2}) ~ [\check{\Sigma}]^2
\end{equation*}
while we also have  the formula 
  \begin{equation*}
\tau (Y)  = \frac{1}{3} (1-p^{-2}) [\check{\Sigma} ]^2  +  \frac{1}{12\pi^2}\int_Y \left(
  |W_+|^2 - |W_-|^2
 \right)_{\check{g}} d\mu_{\check{g}}
\end{equation*}
by assumption. 
Remembering that $[\check{\Sigma}]^2 = q[\Sigma ]^2$, we therefore obtain  
   \begin{equation*}
\tau (Y) - \frac{1}{3} \left(1-\left[\frac{q}{p}\right]^2\right) [{\Sigma} ]^2  =   q\cdot   \frac{1}{12\pi^2}\int_Y \left(
  |W_+|^2 - |W_-|^2
 \right)_{\check{g}} d\mu_{\check{g}}
\end{equation*}
by straightforward algebraic manipulation. 
Reinterpreting the right-hand side as a  curvature integral on $M$ for the 
 edge cone metric $g=\varpi^*\check{g}$, we thus deduce 
 \eqref{prop2} for this large class of   examples
 where $\beta=q/p$ is an arbitrary positive rational number. 
 
 We now consider the general case of 
 Theorem \ref{thing2}
 with $\beta=q/p$ rational.  First,  notice that 
the same elementary trick used in \S \ref{curvy} shows that the
the $\beta$-dependent correction term in  \eqref{prop2}
can be localized to a neighborhood of $\Sigma$. Moreover,
 this correction term is additive
under disjoint unions, and multiplicative under covers. 
 Moreover, Lemma
\ref{pram} shows that they are independent of the particular choice of metric,
so we may assume that the given edge-cone metric of cone angle $2\pi q/p$ 
is rotationally invariant 
 on a tubular neighborhood of $\Sigma$. 
  Cut out such a tubular neighborhood $\mathcal{U}$ 
 of $\Sigma$, and consider the isometric  $\ZZ_q$-action corresponding to 
 rotation in the normal bundle of $\Sigma$ though an angle  of $2\pi/q$.  
  We then have an induced {\em free} action on the oriented $3$-manifold  
$\partial \mathcal{U}$. However, the cobordism group for free $\ZZ_q$-actions is of
finite order in any odd dimension \cite{conflo}. Thus, there is an
 oriented $4$-manifold-with-boundary $Z$ with free $\ZZ_q$-action,
 where $\partial Z$ is 
 disjoint union of, say,  $\ell$ copies of  $\partial \mathcal{U}$, each equipped with the original
  $\ZZ_q$-action. Let $\widetilde{M}$ then be obtained from the reverse-oriented manifold
  $\overline{Z}$ by capping off each of its $\ell$ boundary components with a copy of
  $\mathcal{U}$, and notice that, by construction, $\widetilde{M}$ comes equipped
  with a  $\ZZ_q$-action whose fixed-point set consists of $\ell$ copies of $\Sigma$. 
  We extend $\ell$ copies of the given edge-cone metric $g$ on $\mathcal{U}$ to a 
  $\ZZ_q$-invariant metric $\tilde{g}$ on   $\widetilde{M}$. However, 
  $Y=\tilde{M}/\ZZ_q$ is now a manifold, and $\tilde{g}$ pushes down 
to $Y$ as an orbifold metric $\check{g}$ of cone angle $2\pi/p$.
Our previous argument for global quotients then shows that 
\eqref{prop1} and \eqref{prop2} hold for $(Z,\coprod_1^\ell\Sigma)$,
and additivity therefore shows that the correction terms are as promised
for each of the $\ell$  identical copies of $\mathcal{U}$. This shows  that 
\eqref{prop2} holds for any edge-cone manifold of cone angle $2\pi\beta$, provided 
that $\beta$ is a positive rational number $q/p$. Multiplicativity
of the correction 
under covers similarly allows one to drop the assumption that $\Sigma$ is orientable.

Finally, an  elementary  continuity argument  allows us to extend our formula from 
rational to real $\beta$. Consider a smooth family
of edge-cone metrics on $(M, \Sigma)$ with $\mathbf{SO}(2)$ symmetry about $\Sigma$, but with 
cone angle varying over  the entire positive reals $\RR^+$. 
 The left- and right-hand sides of  \eqref{prop2} then vary
continuously as $\beta$ varies, and their difference vanishes for 
$\beta \in   \RR^+ \cap \QQ$. By continuity, the two sides  therefore agree for all $\beta > 0$. 
With the help of  Lemma  \ref{pram}, this gives us  alternative proofs of 
Theorem \ref{thing2} and Proposition \ref{thing3}. 


\pagebreak 
 
\section{Edges and Einstein  Metrics} \label{hti} 

As we saw in \S \ref{curvy}, Theorems \ref{thing1} and \ref{thing2} are equivalent to 
the fact that every edge-cone metric $g$ of cone angle $2\pi \beta$ on $(M,\Sigma)$ satisfies 
\begin{equation}
\label{dacapo}
(2\chi \pm 3\tau) (M) -  \dft_\pm(\Sigma, \beta) 
=  \frac{1}{4\pi^2}\int_M \left[
 \frac{s^2}{24} + 2|W_\pm|^2  -\frac{|\mathring{r}|^2}{2}
 \right]_g d\mu_g
\end{equation}
for both choices of the $\pm$ sign, where
$$
\dft_\pm (\Sigma,\beta) = 2 (1-\beta)\chi(\Sigma) \pm (1-\beta^2 ) [\Sigma ]^2~.
$$
However, if the edge-cone metric $g$ is Einstein, it then satisfies $\mathring{r}\equiv 0$,
and  the integrand on the right-hand-side of (\ref{dacapo}) is consequently non-negative.
The existence of an Einstein edge-cone metric of cone angle $2\pi \beta$ on
$(M,\Sigma)$ therefore implies the topological constraints
that 
$$
(2\chi \pm 3\tau) (M) \geq  \dft_\pm(\Sigma, \beta) 
$$
and equality can occur for a given choice of sign only if the Einstein metric
satisfies $s\equiv 0$  and $W_\pm\equiv 0$. Since the metric
in question also satisfies $\mathring{r}\equiv 0$ by hypothesis, equality only occurs if 
$\Lambda^\pm$ is flat, which is to say that the metric in question 
is locally hyper-K\"ahler, in a manner compatible with the $\pm$-orientation of $M$. 
This proves a  more refined version of Theorem \ref{oomph}:

\begin{thm} \label{capstone}
Let $M$ be a smooth compact $4$-manifold, and let $\Sigma \subset M$ be a
 compact orientable embedded surface. If $(M,\Sigma)$ admits
an Einstein edge-cone metric of cone angle $2\pi \beta$,
then $(M,\Sigma )$ must satisfy the  inequalities 
\begin{equation}
\label{oompa}
(2\chi + 3\tau) (M) \geq (1-\beta ) \left[ 2\chi (\Sigma)  + (1+\beta ) [\Sigma]^2 \right]
\end{equation}
and 
\begin{equation}
\label{loompa}
(2\chi - 3\tau) (M) \geq (1-\beta ) \left[ 2\chi (\Sigma)  - (1+\beta ) [\Sigma]^2 \right]~.
\end{equation}
Moreover, equality occurs in (\ref{oompa}) if and only if $g$ is locally hyper-K\"ahler, 
in a manner compatible with the given orientation of $M$. Similarly, 
 equality occurs in (\ref{loompa}) if and only if $g$ is locally hyper-K\"ahler, 
in a manner compatible with the opposite orientation of $M$. 
\end{thm}

Recall  that
a Riemannian 
$4$-manifold  is locally hyper-K\"ahler iff it
is Ricci-flat and locally K\"ahler. Brendle's recent construction \cite{brendedge} 
of Ricci-flat K\"ahler manifolds with edge-cone singularities of small cone angle
thus provides an interesting class of examples which saturate inequality (\ref{oompa}). 

As a related illustration of the meaning of Theorem \ref{capstone}, let us now consider
what happens when the cone angle tends to zero.

\begin{cor}
\label{chorus}
If $(M,\Sigma)$ admits
a sequence $g_j$ of  Einstein edge-cone metrics  with cone angles $2\pi \beta_j\to 0$, 
then $(M,\Sigma )$ must satisfy the two inequalities 
$$
(2\chi \pm 3\tau) (M) \geq  2\chi (\Sigma)  \pm  [\Sigma]^2 
$$
with equality for a given sign iff the $L^2$ norms of both $s$ and $W_\pm$   tend to zero
as $j\to \infty$.
\end{cor}

For example, let $\Sigma \subset \CP_2$ be a smooth cubic curve, and 
let $M$ be obtained from $\CP_2$ by blowing up $k$ points
which do not lie on $\Sigma$. Considering  $\Sigma$ as 
a submanifold of $M$, we always have $2\chi(\Sigma ) +[\Sigma ]^2 =0+3^2=9$, whereas 
\linebreak 
$(2\chi + 3\tau )(M) = (2\chi + 3\tau )(\CP_2) -k=9-k$. Thus, 
$(M,\Sigma )$ does not admit Einstein metrics of small cone angle 
when $k$ is a positive integer. By contrast, 
\cite{jmredge} leads one to believe that $(\CP_2, \Sigma)$  {\em should} 
admit  Einstein metrics of small cone angle, and that, 
 as the angle tends to zero, these should  tend to 
previously discovered hyper-K\"ahler metrics \cite{bkhedge,tyhedge}.
This would nicely illustrate the boundary case of Corollary \ref{chorus}.

While many interesting results have recently been obtained 
 about the K\"ahler case of Einstein
edge metrics with $\beta \in (0,1]$, the large cone-angle regime of the problem seems
technically 
intractable. It is thus interesting to observe  Theorem \ref{capstone} gives us strong
obstructions to the existence of Einstein edge-cone metrics with large cone angle, 
even without imposing the K\"ahler condition. Indeed, first notice that 
dividing (\ref{oompa}) and (\ref{loompa}) by $\beta^2$ and taking the limit as 
$\beta \to \infty$ yields the inequalities 
$$ 0 \geq - [\Sigma ]^2 \mbox{ and }  0 \geq  [\Sigma ]^2 $$
so that existence is obstructed for large $\beta$ unless $[\Sigma ]^2 =0$.
Similarly, dividing the sum of (\ref{oompa}) and (\ref{loompa}) by $4\beta$
and letting $\beta \to \infty$  yields
$$ 0 \geq - \chi (\Sigma )$$
so existence for large $\beta$ is obstructed  in most cases:

\begin{cor}
\label{chorale}
Suppose that $\Sigma \subset M$ is a
 connected 
 oriented surface  with  either  non-zero self-intersection or
 genus $\geq 2$. Then there is a real number $\beta_0$ such that
$(M,\Sigma )$ does not admit Einstein edge-cone metrics of cone angle 
$2\pi \beta$ for any $\beta \geq \beta_0$. 
\end{cor}

Of course, 
this result  does not provide obstructions in all cases, and examples show that
this is  inevitable. For example, consider an equatorial $2$-sphere 
 $S^2\subset S^4$. If $S^4$ is thought of as the unit sphere in  $\RR^5 = \RR^3\times \RR^2$, we make take our $S^2$ to be the inverse image of the origin under projection to 
 $\RR^2$. By taking polar coordinates on this $\RR^2$, we can thus identify $S^4-S^2$ with 
$B^3\times S^1$. Rotating in the circle factor then gives us a Killing field on $S^4$, and there 
is  an interesting  conformal rescaling of the standard metric obtained by requiring
that this conformal Killing field  have unit length in the new metric. What we 
obtain in this way is a conformal equivalence between $S^4-S^2$
and the Riemannian product $\mathcal{H}^3\times S^1$. More precisely, 
the standard metric on $S^4$ now becomes
$$(\sech^2 \rad )[h +  d\theta^2 ]$$
where $h$ is the standard curvature $-1$ metric on  $\mathcal{H}^3$, and where 
$\rad : \mathcal{H}^3\to \RR$ is the distance in $\mathcal{H}^3$ from some arbitrary 
base-point. By a minor alteration, we then obtain the family
$$g= (\sech^2 \rad )[h + \beta^2 d\theta^2 ]$$
of edge-cone metrics on $S^4$ with arbitrary cone angle $2\pi\beta$. By setting
$\tilde{\theta}= \beta\theta$, we see that these edge-cone metrics are actually  locally
isometric to the standard metric on $S^4$, and so, in particular, are all Einstein;
in other words, these edge-cone metrics are simply obtained by passing to the universal cover 
$B^3\times \RR$ of $S^4-S^2$, and then dividing out by some arbitrary translation of
the $\RR$ factor. 
 Since this works for any  $\beta >0$,  we see that it is inevitable that 
 Corollary \ref{chorale} does not apply to genus zero surfaces of trivial self-intersection.

The above edge-cone metrics $g$ can be obtained from the family 
$$g_0= \beta^{-1} h + \beta~d\theta^2$$
by  a suitable conformal rescaling. In the next section, 
we will see that this can be interestingly generalized by replacing the
constant function $V=\beta^{-1}$ with a harmonic function, and 
by replacing the flat circle bundle ${\mathcal H}^3\times S^1$ with 
a principal ${\mathbf U}(1)$-bundle over ${\mathcal H}^3$ which is  equipped  with a connection whose curvature is the closed $2$-form 
$\star dV$.

  \pagebreak 
   
\section{Edges and Instantons} \label{zoo}

In this section, we will study families of   self-dual edge-cone metrics
on  $4$-manifolds,   and observe that these metrics 
are interestingly related to certain  gravitational instantons.
For our purposes, a gravitational instanton will mean a complete  non-compact
Ricci-flat Riemannian $4$-manifold which is both simply connected and self-dual, 
in the sense that 
$W_-=0$. Note that such spaces are necessarily hyper-K\"ahler, but that the orientation
we will give them here is opposite the one induced by the hyper-K\"ahler structure. 

As indicated at the end of \S \ref{hti}, we will 
 begin by considering a  construction  \cite{mcp2}, called the {\em hyperbolic ansatz}, 
that builds explicit self-dual conformal metrics out of positive harmonic functions on 
regions of hyperbolic $3$-space. To this end, let 
${\mathcal U}\subset {\mathcal H}^3$ be an open set in hyperbolic $3$-space,
and let $V: {\mathcal U}\to \RR^+$ be a function which is harmonic with respect
to the hyperbolic metric $h$. The $2$-form $\star dV$ is then closed, and
we will furthermore suppose that the deRham class $[(\star dV)/2\pi]$
represents an integer class in $H^2({\mathcal U}, \RR)$. The  theory of Chern classes 
then guarantees that there is a principal $\mathbf{U} (1)$-bundle ${\mathcal P}\to {\mathcal U}$ 
which carries connection $1$-form $\Theta$ of curvature
$d\Theta = \star dV$.  We may then consider the Riemannian metric 
\begin{equation}
\label{ansatz}
g_0= Vh+ V^{-1}\Theta^2
\end{equation}
on the total space $\mathcal P$ of our circle bundle. Remarkably, this metric
is automatically self-dual with respect to a standard orientation of ${\mathcal P}$. 
Since this last condition, that $W_-=0$, is conformally invariant,  multiplying 
 $g_0$ by any positive conformal factor will result in  another self-dual metric. 
 For shrewd choices of $V$ and the conformal factor,  interesting
compact self-dual edge-cone manifolds can be constructed in this way. Indeed, 
 one can even  sometimes arrange for  the resulting edge-cone metric to also be  Einstein. 
 
 We already considered the case of constant $V$ at the  end of  \S \ref{hti}. 
To obtain something more interesting, we now  choose our  potential to be  
\begin{equation}
\label{choice} 
V= \beta^{-1} + \sum_{j=1}^n G_{p_j}
\end{equation}
where  $\beta$ is an arbitrary positive constant, 
 $p_1, \ldots , p_n\in {\mathcal H}^3$ are distinct points in hyperbolic $3$-space,
and where $G_{p_j}$ are the corresponding Green's functions. 
For simplicity, we will use the same  conformal rescaling 
$$g= \beta (\sech^2 \rad ) g_0$$
that was used in \S \ref{hti}, where $\rad$ denotes the distance from some arbitrary
base-point in $\mathcal{H}^3$. The metric-space completion 
$M=\mathcal{P}\cup \Sigma\cup \{ \hat{p}_j\}$
of 
$(\mathcal{P},g)$ then carries a natural smooth structure making it diffeomorphic to 
the connected sum 
$$n\CP_2= \underbrace{\CP_2 \# \cdots \# \CP_2}_n$$
and  
$g$ then extends to $M$ 
as an edge-cone metric 
with cone angle $2\pi\beta$ along a surface $\Sigma \approx S^2$
of self-intersection $n$. Indeed,  $\beta=1$, 
this exactly reproduces the  self-dual metics on $n\CP_2$ constructed  in \cite{mcp2}.
For general $\beta$, the picture is essentially the same; metric-space completion 
adds one point $\hat{p}_j$ for each of the base-points $p_j$, and a $2$-sphere $\Sigma$ corresponding
to the $2$-sphere at infinity of $\mathcal{H}^3$. By the same argument used in 
 \cite[p. 232]{mcp2},  the metric $g$ extends smoothly across
the $\hat{p}_j$. By contrast,  we obtain an edge-cone metric of cone-angle $\beta$ along 
$\Sigma$,  because our potential $V$ is asymptotic to the constant choice considered
in \S \ref{hti}.  
Since  these edge-cone metrics satisfy 
$W_-=0$, Theorem \ref{thing2}  tells us that 
they also satisfy
\begin{equation}
\label{zest}
\frac{1}{12\pi^2}\int_{n\CP_2} |W_+|^2d\mu = \tau (n\CP_2) - \frac{1}{3}(1-\beta^{2}) 
[\Sigma ]^2=
\frac{n(2+\beta^2)}{3}~.
\end{equation}

The  $n=1$  case has some special features that make it  particularly interesting. The potential 
becomes 
\begin{equation}
\label{pot}
V= \beta^{-1} + \frac{1}{e^{2\rad}-1}
\end{equation}
and  \eqref{ansatz} can then  be written explicitly as 
\begin{equation}
\label{voila}
g_0= V \left[ d\rad^2 + (4\sinh^2 \rad) (\sigma_1^2 + \sigma_2^2)\right] + 
 V^{-1} \sigma_3^2
\end{equation}
where  $\rad$ represents the hyperbolic distance from the point 
$p=p_1$, and $\{\sigma_j \}$ is a left-invariant orthonormal co-frame
for $S^3 = \mathbf{SU}(2)$. 
Remarkably, the alternative representative 
\begin{equation}
\label{recipe}
\tilde{g}= 4\beta^{-1} [(2-\beta)  \cosh \rad +\beta  \sinh \rad ]^{-2}g_0
\end{equation}
of the conformal class is then Einstein, with Einstein constant $\lambda= \frac{3}{2}\beta^2 (2-\beta)$,
as  follows   from  \cite[Equation (2.1)]{shin} or \cite[\S 9]{hitpain}. 
Demanding that  $g$  define a metric for all $\rad\geq 0$ imposes the 
constraint\footnote{
By contrast, when $\beta > 2$, restricting  (\ref{recipe})  to the  $4$-ball  
 $\rad < \tanh^{-1}(1-2/\beta)$ 
   produces a family of complete self-dual Einstein metrics originally discovered by 
Pedersen \cite{henrik}.} 
 that $\beta < 2$.  
 The resulting family of edge-cone metrics on $(\CP_2, \CP_1)$
 of cone angle $2\pi \beta$,  $\beta\in (0,2)$,     coincides with the 
family constructed  by Abreu \cite[\S 5]{abreu} by  an entirely different method. 
Notice that, as a special case of \eqref{zest}, these metrics satisfy  
\begin{equation}
\label{arcturus}
\frac{1}{12\pi^2} \int_{\CP_2} |W_+|^2d\mu = 
\frac{2+\beta^2}{3}~
\end{equation}
whether we represent the conformal class by $g$ or,  when $\beta\in (0,2)$, by $\tilde{g}$. 

When $\beta=1$, $\tilde{g}$  is just the standard Fubini-Study metric
on $\CP_2$. On the other hand, the  $\beta \to 0$ and  $\beta \to 2$ limits 
of $\tilde{g}$ give us  
other celebrated metrics. For example, introducing the new radial coordinate
 $\mathfrak{r}=\sqrt{\coth \rad}$, 
 \begin{equation}
\label{gucci}
\lim_{\beta\to 2}  \tilde{g}= \frac{d{\mathfrak{r}}^2}{1-{\mathfrak{r}}^{-4}}+ {\mathfrak{r}}^2
\left[ \sigma_1^2+\sigma_2^2 + (1-{\mathfrak{r}}^{-4})\sigma_3^2
\right]
\end{equation}
which is the usual formula for the  Eguchi-Hanson metric \cite{eh,gibhawk}, a celebrated
complete self-dual Einstein metric on the manifold $TS^2$;  however, 
 (\ref{gucci})  has arisen here as a metric on $\mathbf{SU}(2)\times (1,\infty )$
rather than on  $\mathbf{SO}(3)\times (1,\infty )$, so this  version of Eguchi-Hanson is actually
actually a branched double cover of the usual one, ramified along the zero section 
of $TS^2$. 
On the other hand, after introducing  a new radial coordinate $\mathfrak{r}= \beta^{-1}\rad$,
the point-wise coordinate limit 
$$
\lim_{\beta\to 0}\tilde{g}= \left(1+\frac{1}{2\mathfrak{r}}\right)
\left[d\mathfrak{r}^2+ 4 \mathfrak{r}^2(\sigma_1^2+\sigma_2^2)\right]+ 
\left(1+\frac{1}{2\mathfrak{r}}\right)^{-1}\sigma_3^2
$$
 is  the  Taub-NUT metric \cite{gibhawk,lebnut},
a complete non-flat hyper-K\"ahler metric on $\RR^4$. 
Similarly, by 
choosing suitable sequence of centers $\{p_j\}$ and conformal rescalings,
there are $\beta\to 0$ limits of our conformal metrics on $n\CP_2$ which converge to 
\begin{equation}
\label{multitn}
g_{\rm{multi}} = \tilde{V} \mathbf{dx}^2 + \tilde{V}^{-1}\tilde{\Theta}^2
\end{equation}
where $\mathbf{dx}^2$ is the Euclidean metric on $\RR^3$, the harmonic function 
$$
\tilde{V} = 1+ \sum_{j=1}^n \frac{1}{2\mathfrak{r}_j}
$$
is expressed in terms of the Euclidean distances $\mathfrak{r}_j$
from the $\tilde{p}_j$, and where $d\tilde{\Theta}= \star d\tilde{V}$. 
The metric \eqref{multitn}  is a famous gravitational instanton called  the 
multi-Taub-NUT metric \cite{gibhawk,lebnut,minerbe}.

If we are cavalier about interchanging integrals and limits, these obseravtions
provide some interesting  information regarding the above gravitational instantons.
For example,  the standard Eguchi-Hanson space 
 ${\mathbb{EH}}$ should  satisfy
$$\frac{1}{12\pi^2}\int_{\mathbb{EH}} |W_+|^2d\mu =\lim_{\beta\to 2^-}  \frac{1}{2}  \left(
\frac{2+\beta^2}{3}\right) = 1$$
with respect to the orientation for which $W_-=0$; 
here the factor of $1/2$ stems from the fact that $\mathbb{EH}$ is a 
$\ZZ_2$-quotient of 
$\CP_2 -\{ p \}$. 
Because   the $s$, $W_-$ and $\mathring{r}$ pieces of the curvature
 tensor ${\mathcal R}$  all vanish for $\mathbb{EH}$, the  $L^2$ norm squared of the 
 curvature  of the Eguchi-Hanson instanton should therefore be given by 
 \begin{equation}
\label{ehcurv}
\int_{\mathbb{EH}}|{\mathcal R}|^2 d\mu = \int_{\mathbb{EH}} |W_+|^2d\mu = 12\pi^2.
\end{equation}
Similarly, 
the Taub-NUT gravitational instanton $\mathbb{TN}$  should satisfy 
$$\frac{1}{12\pi^2}\int_{\mathbb{TN}} |W_+|^2d\mu = 
 \lim_{\beta\to 0^+}[\tau (\CP_2) - \frac{1}{3}(1-\beta^{2}) [\CP_1]^2]= \frac{2}{3}$$
when oriented so that  $W_-=0$. In particular, the  $L^2$ norm squared of the 
 curvature tensor  of the Taub-NUT space is given by 
 \begin{equation}
\label{tnutcurv}
\int_{\mathbb{TN}}|{\mathcal R}|^2 d\mu = \int_{\mathbb{TN}} |W_+|^2d\mu = 8\pi^2.
\end{equation}
Similar  thinking predicts that the $n$-center multi-Taub-NUT metric 
will have 
\begin{equation}
\label{nuttier}
\int |\mathcal{R}|^2 d\mu = \int |W_+|^2d\mu = 
{12\pi^2} \lim_{\beta\to 0^+}
\frac{n(2+\beta^2)}{3} = 8\pi^2n~.
\end{equation}
All of these interchanges of integrals and limits  can in fact  be  rigorously justified
by a careful application of the  the dominated convergence theorem.  
Rather than presenting  all  the tedious details, however, we will instead simply 
 double-check these answers later, using a more direct method.


The self-dual Einstein edge-cone metrics on $(\CP_2, \CP_1)$ given by 
(\ref{recipe}) 
are  invariant under the   action of $\mathbf{SU}(2)$ on $\CP_2= \CC^2\cup \CP_1$.
However, there is  a remarkable second family of self-dual Einstein edge-cone metrics
on $(\CP_2, \RP^2)$ which are invariant under a different action of
$\mathbf{SU}(2)$ on $\CP_2$, namely the action of 
$\mathbf{SO}(3)=\mathbf{SU}(2)/\ZZ_2$ on $\CP_2$ gotten by thinking
of it as the projective space of  $\CC  \otimes \RR^3$. 
These metrics, which were
discovered by Hitchin \cite{hitedge1,hitedge2,hitpain}. The starting point of Hitchin's investigation 
was a  reduction, due to Tod \cite{todpain},
 of the self-dual Einstein  equations with $\mathbf{SU}(2)$
symmetry to an ordinary differential equation belonging to the  Painlev\'e VI family.  
 He then finds a specific family of solutions depending on an integer $k\geq 3$
 which have the property that the corresponding twistor spaces are (typically singular) algebraic
 varieties, and then shows  \cite[Proposition 5]{hitedge1} that the resulting Einstein manifold compactifies as an edge-cone metric on $(\CP_2, \RP^2)$ of cone angle $4\pi /(k-2)$; he phrases this assertion 
 in terms of the 
 $\ZZ_2$-quotient of this metric by complex conjugation, which is then an orbifold metric on $S^4$ 
 with an edge-cone singularity of cone-angle $2\pi/(k-2)$ along  a Veronese $\RP^2\subset S^4$. 
 When 
 $\beta=1$, Hitchin's metric is just the Fubini-Study metric, while for $\beta=2$
 it is just a branched double cover of the standard metric on $S^4$. 
The corresponding solutions of Painlev\'e VI  can be explicitly expressed in terms 
of elliptic functions, and Hitchin  observes in a later paper  \cite[Remark 2, p. 79]{hitpain} 
that, in principle, 
  solutions for non-integer $k$ should also   give
 rise to  a self-dual Einstein edge-cone metrics on 
 $(\CP_2, \RP^2)$, although he does not try to determine precisely which  cone angles 
 $2\pi\beta$ 
 can 
 actually arise in this way. However, since $\RP^2\subset \CP_2$ is  totally real, 
 Proposition \ref{thing3} is perfectly adapted to the study of Hitchin's self-dual edge-cone 
 metrics, and tells
 us that they necessarily satisfy  
$$
 \int_{\CP_2} |W_+|^2 d\mu = 12\pi^2  [\tau (\CP_2) +\frac{1}{3} (1-\beta^2) \chi (\RP^2)] = 4\pi^2 (4-\beta^2) ~.
$$
 We thus see  that the constraint $\beta \leq 2$, corresponding to $k\geq 3$,
  is both natural and  unavoidable.

 The $\mathbf{SU}(2)$-action on $\CP_2$  which preserves Hitchin's metrics  
 has two $2$-dimensional orbits, namely the $\RP^2$ where the edge-cone singularity occurs, 
 and a conic $C\subset \CP_2$ where the metric is smooth; in the 
 $S^4= \CP_2/\ZZ_2$ model,
 $C$ projects to another  Veronese $\RP^2=C/\ZZ_2$. 
  Hitchin now normalizes his  metrics so that $C$ has area $\pi$ for each 
  $\beta$, and asks what happens when $\beta= 2/(k-2) \to 0$. He then shows \cite[Proposition 6]{hitedge1}
  that this 
  $k\to \infty$ limit is precisely the Atiyah-Hitchin gravitational instanton.
  Here we need to  be rather precise, because there are really two different versions of the
  Atiyah-Hitchin instanton.
  The better known version, which we shall call $\mathbb{AH}$, was constructed 
   \cite{ahso} as the moduli space $M^0_2$ of   solutions of the $\mathbf{SU}(2)$ Bogomolny 
monopole equations on $\RR^3$ with magnetic charge $2$  and 
fixed center of mass. With this convention, $\mathbb{AH}$ is diffeomorphic to a
tubular neighborhood of the Veronese $\RP^2$ in $S^4$, and so 
has  fundamental group $\ZZ_2$. Its universal cover $\widetilde{\mathbb{AH}}$ 
is therefore also a gravitational instanton, and is diffeomorphic to a tubular neighborhood
of a conic $C$ in $\CP_2$; thus, $\widetilde{\mathbb{AH}}$ is diffeomorphic   both to
$\CP_2-\RP^2$ and to the ${\mathcal O}(4)$ line bundle
over $\CP_1$. If we regard Htichin's metrics as edge-cone metrics on $(\CP_2,\RP^2)$,
 he then shows that they converge to 
$\widetilde{\mathbb{AH}}$ as $k\to\infty$. 
Interchanging limits and integration as before thus leads us to expect that 
$$\int_{\widetilde{\mathbb{AH}}}  |W_+|^2 d\mu = \lim_{\beta_k\to 0^+} 4\pi^2 (4-\beta_k^2) = 
16\pi^2$$
so that 
\begin{equation}
\label{ahtildecurv}
\int_{\widetilde{\mathbb{AH}}}  |\mathcal{R}|^2 d\mu = 16\pi^2~,
\end{equation}
and hence that 
\begin{equation}
\label{ahcurv}
\int_{{\mathbb{AH}}}  |\mathcal{R}|^2 d\mu = \frac{1}{2}\int_{\widetilde{\mathbb{AH}}}  |\mathcal{R}|^2 d\mu= 8\pi^2~.
\end{equation}
While this argument  can  again be made rigorous using the dominated convergence theorem, 
we will instead simply double-check these answers 
by a second
method which fits into a beautiful general pattern. 

So far, we have been using  the signature formula \eqref{prop2} to compute the 
$L^2$-norm of the  self-dual Weyl curvature for interesting edge-cone metrics, 
and then, by a limiting process,  have inferred the  $L^2$-norm of the Riemann curvature 
for various 
gravitational instantons, as $W_+$ is the only non-zero piece of the curvature 
tensor for such spaces. However, we could have instead proceeded by considering the Gauss-Bonnet
formula \eqref{prop1} for edge-cone metrics. In this case, we have 
$$
\lim_{\beta\to 0}\int_M\left[|\mathcal{R}|^2-|\mathring{r}|^2 \right]   d\mu 
= 8\pi^2 \lim_{\beta\to 0} \left[
\chi (M) -(1-\beta)\chi (\Sigma) 
\right] = 8\pi^2 \chi (M -\Sigma)~.
$$
Since the limit metric is Einstein, we thus expect any gravitational instanton
$(X,g_\infty)$ obtained as a $\beta\to 0$ limit to  satisfy
\begin{equation}
\label{pattern}
\int_{X}|\mathcal{R}|^2d\mu = 8\pi^2 \chi (X)
\end{equation}
because the underlying  manifold of the instanton is $X=M-\Sigma$.
For example, the simply-connected Atiyah-Hitchin space $\widetilde{\mathbb{AH}}$
deform-retracts to $S^2$, and so has Euler characteristic $2$; thus, 
its total squared curvature should,  as previously inferred,  be $16\pi^2$. 
Similarly, the Taub-NUT metric lives on the contractible space $\RR^4$, which has Euler
characteristic $1$, and so is expected to have $\int |\mathcal{R}|^2d\mu = 8\pi^2$, 
in agreement with our previous inference. 

These same answers can be obtained in a 
direct  and 
rigorous manner  by 
means of the Gauss-Bonnet theorem with boundary.
Indeed, if  $(Y,g)$ is any compact oriented Riemannian 
$4$-manifold-with-boundary, this result tells us 
\begin{equation}
\label{gbound}
\chi (Y) = \frac{1}{8\pi^2} \int_Y \Big( |{\mathcal R}|^2 - |\mathring{r}|^2\Big) \mu_g +  \frac{1}{4\pi^2}
 \int_{\partial Y}  \Big[ 2 \det (\gemini )+ \langle \gemini , \hat{\mathcal R}\rangle  \Big] d{\zap a}
\end{equation}
where $\gemini$ and $d{\zap a}$ are the
 the second fundamental form and  volume $3$-form of the boundary $\partial Y$, 
 is the of the boundary, and $\hat{\mathcal{R}}$ is the symmetric  tensor field on 
$\partial Y$  gotten by restricting the ambient curvature tensor $\mathcal{R}$
and then 
using the   $3$-dimensional  Hodge star operator
to  identify $\odot^2\Lambda^2$  and $\odot^2\Lambda^1$ on this $3$-manifold. 
One can prove \eqref{gbound} simply by following  Chern's  proof of the 
generalized Gauss-Bonnet theorem, using Stokes theorem to count
the zeroes, with multiplicities, of a generic vector field on $Y$  that is an outward pointing 
normal field along  $\partial Y$.  The proof in  \cite{chern} then goes through without changes, except that 
there is now a non-trivial contribution due to $\partial Y$. 

The best-known class of gravitational instantons consists of 
ALE (Asymptotically Locally Euclidean) spaces. For such a space $X$, there is 
a compact set $K$ such that $X-K$ is diffeomorphic to $(\RR^4-B)/\Gamma$
for some finite subgroup $\Gamma \subset \mathbf{SU}(2)$, in such a manner that 
the metric is given by 
$$
g_{jk}= \delta_{jk} + O(\radius^{-2})
$$
where $\radius$ is the Euclidean radius, with coordinate derivatives 
$\partial^kg$ commensurately falling off like $\radius^{-2-k}$.
In particular, a ball of radius $\radius$ has $4$-volume $\sim \radius^4$, 
and the Riemannian curvature falls off like  $|\mathcal{R}|\sim \radius^{-4}$. 
Moreover, the hypersurface $\radius=\text{const}$  has $3$-volume $\sim \radius^3$
and $|\gemini|\sim \radius^{-1}$.  
 If we let $Y\subset X$ be the region $\radius\leq C$ and then 
let $C\to \infty$, we thus see that the only significant boundary contribution 
in \eqref{gbound} comes from the $\det \gemini$  term. In the limit, 
 it is thus easy to show  that the boundary terms  just 
equals  $1/|\Gamma|$  times the corresponding integral for the
  standard $3$-sphere  $S^3\subset \RR^4$. 
Thus, any ALE instanton satisfies \cite{chenlebweb,kasue,nakajale}
\begin{equation}
\label{alecurv}
\int_X |\mathcal{R}|^2d\mu = 8\pi^2 \left(\chi(X) -\frac{1}{|\Gamma|}\right)~.
\end{equation}
For example, when $X$ is the Eguchi-Hanson instanton $\mathbb{EH}$,
$\chi = 2$ and $|\Gamma|=2$, so the total squared curvature is $12\pi^2$, 
as previously predicted by \eqref{ehcurv}.

The Taub-NUT space is the prototypical example of an asymptotically
locally flat (ALF) gravitational instanton. 
For such spaces, the volume of a large ball has $4$-volume $\sim \radius^{3}$,
while curvature falls off like $\radius^{-3}$. The hypersurface $\radius = \text{const}$ 
 has $3$-volume $\sim \radius^{2}$,  with $|\gemini|\sim \radius^{-1}$ and $\det \gemini\sim \radius^{-4}$. 
Thus the boundary contribution in \eqref{gbound} tends to zero as $\radius\to \infty$, and 
 any ALF instanton therefore satisfies   the simpler formula  \cite{daiwei,kasue} 
\begin{equation}
\label{alfcurv}
\int_X |\mathcal{R}|^2d\mu = 8\pi^2 \chi(X)
\end{equation}
previously seen  in \eqref{pattern}. 
For example, the Taub-NUT instanton has $\chi=1$, so its total squared curvature
$8\pi^2$, as predicted by \eqref{tnutcurv}. The Atiyah-Hitchin gravitational instanton is also
ALF; in fact, it is 
asymptotic to Taub-NUT with a negative NUT parameter \cite{ahso}. Thus $\widetilde{\mathbb{AH}}$, with an Euler characteristic of $2$, has total squared curvature $16\pi^2$,
while its $\ZZ_2$-quotient $\mathbb{AH}$, with an Euler characteristic of $1$, has 
has total squared curvature $8\pi^2$. Note that these conclusions coincide 
with the predictions of \eqref{ahtildecurv} and \eqref{ahcurv}.

A classification of complete hyper-K\"ahler ALE $4$-manifolds was given by 
Kronheimer \cite{kron,krontor}; 
up to deformation, they are in one-to-one correspondence with 
the Dynkin diagrams of type $A$, $D$, and $E$. For example, the 
Dynkin diagram $A_1$ corresponds to the
Eguchi-Hanson metric, and,   more generally, 
 the  Dynkin diagrams $A_k$ correspond to 
the multi-Eguchi-Hanson metrics independently  discovered by 
Hitchin \cite{hitpoly} and by Gibbons-Hawking \cite{gibhawk}. Each Dynkin diagram
represents a discrete subgroup $\Gamma$ of $\mathbf{SU}(2)$, and these groups are 
then realized as the fundamental group of the $3$-dimensional boundary 
 at infinity of the corresponding  $4$-manifold. 
 The diagram also elegantly  encodes the diffeotype of the corresponding $4$-manifold, 
 which is obtained by plumbing together copies of $TS^2$, with one $2$-sphere
 for each node, and with edges of the diagram indicating which   pairs of $2$-spheres 
 meet. In particular, the number $k$ of nodes in any given diagram is the second Betti number
 $b_2$ of the instanton, which consequently has $\chi=k+1$. 
 
 For each diagram of type $A$ or $D$, there is also an associated ALF instanton. 
 These ALF partners 
 are diffeomorphic to the corresponding ALE instantons, but their geometry at infinity 
 resembles Taub-NUT$/\Gamma$ instead of  a Euclidean
 quotient $\RR^4/\Gamma$.
 In the $A_k$ cases, these spaces are just the multi-Taub-NUT metrics of \eqref{multitn},
 with 
 $n=k+1$ centers;  our heuristic calculation \eqref{nuttier}
 of their  total squared curvature thus confirmed by \eqref{alfcurv}, since these spaces
 have  $\chi=n$.
  The $D_k$ metrics were constructed explicitly  
 by Cherkis and Hitchin \cite{cherhit}, building on earlier existence arguments 
  of Cherkis and Kapustin \cite{cherka}. The fact that these metrics really are
  ALF follows from results due to Gibbons-Manton \cite{gibman} and Bielawski \cite{biegibman}. 
Table \ref{bookie} gives a compilation of the total squared curvature of these important  spaces.

Of course, several  of the 
gravitational instantons 
 we have discussed do not  appear on Table \ref{bookie};
 the Taub-NUT  space $\mathbb{TN}$ and the Atiyah-Hitchin manifolds $\mathbb{AH}$ and 
 $\widetilde{\mathbb{AH}}$ are nowhere to be found. Of course, one might decree \cite{cherhit}
 that  $\mathbb{TN}$ is the ALF entry across from the fictitious diagram $A_0$, 
 or that   $\mathbb{AH}$ is the ALF  entry associated with the make-believe  
  diagram $D_0$; but to us, this is less interesting than the general point that 
   that $\int |\mathcal{R}|^2d\mu=
  8\pi^2 \chi$ for complete Ricci-flat $4$-manifolds with ALF asymptotics.

There is a large realm of gravitational instantons  which 
have  slower volume growth, but still have finite topological type. 
Cherkis \cite{cheralh} has proposed sorting these into two classes: the ALG
spaces, with at least quadratic volume growth, and the ALH instantons, 
with sub-quadratic volume growth. For example, 
Tian and Yau \cite{tyhedge} constructed 
 hyper-K\"ahler metrics on the complement of an anti-canonical divisor on any 
 del Pezzo surface; these metrics have volume growth $\sim \radius^{4/3}$, and so 
 are of ALH type. For these examples, one has enough control at infinity
 to see that the boundary term in  \eqref{gbound}  becomes negligeable at large radii, so that 
 the  pattern  $\int |\mathcal{R}|^2d\mu=
  8\pi^2 \chi$ 
continues to  hold. Presumably, this pattern will also turn out to hold for 
all gravitational instantons of type ALG and ALH.  

Many  gravitational instantons do seem to arise   as
 limits of edge-cone Einstein metrics. For instance, 
 the results of \cite{jmredge} strongly suggest that the ALH 
 examples  of  Tian-Yau are 
  $\beta\to 0$
 limits of K\"ahler-Einstein edge-cone metrics;
 proving this, however,  would entail first establishing 
  a  suitable lower bound for the  $K$-energy. It would obviously be
   interesting to prove the existence of 
  sequences of Einstein edge-cone metrics  which tend to other  known examples of gravitational 
  instantons.  
  On the other hand, one might hope  to construct new examples of 
  gravitational instantons as 
   Gromov-Hausdorff limits of suitable sequences of 
  $4$-dimensional Einstein edge-cone manifolds. While some features of  
   weak convergence for smooth Einstein metrics
   \cite{andprior,cheeti} may  carry over 
  with little change, it  seems  likely that the introduction of  
   edge-cone singularities may involve some serious 
   technical difficulties. Our hope is that    the regime of 
   small $\beta$  will  nonetheless prove to 
  be     manageable,    and that  the theory that  emerges will lead to 
   new insights concerning 
   precisely  which  gravitational instantons arise  as $\beta\to 0$ limits of Einstein spaces
  with edge-cone singularities.

     \begin{table}
\begin{center}
  \begin{tabular}{|l | c |c ||  c | c | }
    \hline
     \multicolumn{3}{|c||}{ Group at Infinity}&
        \multicolumn{2}{|c|}{ $\int |\mathcal{R}|^2d\mu$}\\
        \hline
        Dynkin Diagram&
$\Gamma\subset \mathbf{SU}(2)$   & $|\Gamma|$  & ALE & ALF \\ \hline
    $A_k$ \hspace{.7cm}
\begin{minipage}[c]{0.4in}
\setlength{\unitlength}{1ex}
\begin{picture}(50,5)(0,-2.5)
\put(-2,0){\circle*{1}}
\put(0, -0.1){$\ldots$}
\put(4, 0){\circle*{1}}
\put(6, 0){\circle*{1}}
\put(8, 0){\circle*{1}}
\put(-2,0){\line(1,0){1.5}}
\put(3,0){\line(1,0){1}}
\put(4,0){\line(1,0){2}}
\put(6,0){\line(1,0){2}}
\end{picture}
\end{minipage} &cyclic & $k+1$ & $8\pi^2 \left(k+1 -\frac{1}{k+1}\right)$ & $8\pi^2 (k+1)$ \\ \hline
    $D_k$ \hspace{.7cm}
\begin{minipage}[c]{0.4in}
\setlength{\unitlength}{1ex}
\begin{picture}(50,5)(0,-2)
\put(-2,0){\circle*{1}}
\put(0, -0.05){$\ldots$}
\put(4, 0){\circle*{1}}
\put(6, 0){\circle*{1}}
\put(8, 0){\circle*{1}}
\put(-2,0){\line(1,0){1.5}}
\put(3,0){\line(1,0){1}}
\put(4,0){\line(1,0){2}}
\put(6,0){\line(1,0){2}}
\put(6,0){\line(0,1){2}}
\put(6, 2){\circle*{1}}
\end{picture} 
\end{minipage}
&dihedral$^*$&$4k-8$  &$8\pi^2\left(k+1 -\frac{1}{4k-8}\right)$  & $8\pi^2 (k+1)$ \\
    \hline
      $E_6$  \hspace{.7cm}
\begin{minipage}[c]{0.4in}
\setlength{\unitlength}{1ex}
\begin{picture}(50,5)(0,-2)
\put(0,0){\circle*{1}}
\put(4, 2){\circle*{1}}
\put(2, 0){\circle*{1}}
\put(4, 0){\circle*{1}}
\put(6, 0){\circle*{1}}
\put(8, 0){\circle*{1}}
\put(0,0){\line(1,0){2}}
\put(2,0){\line(1,0){2}}
\put(4,0){\line(1,0){2}}
\put(6,0){\line(1,0){2}}
\put(4,0){\line(0,1){2}}
\end{picture} 
\end{minipage}
&tetrahedral$^*$ &$24$  &$8\pi^2\left(7 -\frac{1}{24}\right)$  & --- \\
       $E_7$ \hspace{.7cm}
\begin{minipage}[c]{0.4in}
\setlength{\unitlength}{1ex}
\begin{picture}(50,5)(0,-2)
\put(-2,0){\circle*{1}}
\put(0,0){\circle*{1}}
\put(4, 2){\circle*{1}}
\put(2, 0){\circle*{1}}
\put(4, 0){\circle*{1}}
\put(6, 0){\circle*{1}}
\put(8, 0){\circle*{1}}
\put(-2,0){\line(1,0){2}}
\put(0,0){\line(1,0){2}}
\put(2,0){\line(1,0){2}}
\put(4,0){\line(1,0){2}}
\put(6,0){\line(1,0){2}}
\put(4,0){\line(0,1){2}}
\end{picture} 
\end{minipage} &octohedral$^*$ & $48$ & $8\pi^2\left(8 -\frac{1}{48}\right)$  & --- \\
      $E_8$ \hspace{.7cm}
\begin{minipage}[c]{0.4in}
\setlength{\unitlength}{1ex}
\begin{picture}(50,5)(0,-2)
\put(-4,0){\circle*{1}}
\put(-2,0){\circle*{1}}
\put(0,0){\circle*{1}}
\put(4, 2){\circle*{1}}
\put(2, 0){\circle*{1}}
\put(4, 0){\circle*{1}}
\put(6, 0){\circle*{1}}
\put(8, 0){\circle*{1}}
\put(-4,0){\line(1,0){2}}
\put(-2,0){\line(1,0){2}}
\put(0,0){\line(1,0){2}}
\put(2,0){\line(1,0){2}}
\put(4,0){\line(1,0){2}}
\put(6,0){\line(1,0){2}}
\put(4,0){\line(0,1){2}}
\end{picture} 
\end{minipage}
&dodecahedral$^*$ & $120$ & $8\pi^2\left(9 -\frac{1}{120}\right)$  & --- \\
    \hline
  \end{tabular}
\end{center}
\caption{Total Squared Curvature of  ALE  \&  ALF Gravitational Instantons
   \label{bookie}}
\end{table}


  \pagebreak 

\appendix
\setcounter{section}{0}
\section{Appendix}
\label{apps}

The proof outlined in \S \ref{argument} involves the use of  the signature theorem for orbifolds.  We
will now
 indicate how this can be deduced  from the theory of transversally elliptic operators
 developed in \cite{atgroup}. For further details, see \cite{kawasaki}. 
 
 Let $X$ be a compact manifold equipped with the action of a compact Lie group $\mathbf{G}$. 
A differential operator $D$ between vector bundles on  $X$ is said to be {\em transversally elliptic} if it is  $\mathbf{G}$-invariant,  and its  restriction to any  submanifold transverse to a $\mathbf{G}$-orbit is elliptic.
 A trivial but important example arises when $\mathbf{G}$ acts freely on $X$, so that $X/\mathbf{G}$ is itself a manifold; in this case,  a transversally elliptic operator is essentially just 
  an elliptic operator on $X/\mathbf{G}$.  For us, the case of primary interest occurs when the action
  of $\mathbf{G}$ has only finite isotropy groups. In this case,  $X/\mathbf{G}$ is an orbifold. 

The  index of such a transversally elliptic operator $D$ 
 is an invariant distribution on $\mathbf{G}$.  Equivalently,   the index is given by an infinite series
\begin{equation} \label{at1} 
\ind 
(D)=\sum_\lambda a_\rho \chi_\rho
\end{equation}
 where $\rho$ runs over the irreducible representations of $\mathbf{G}$, $\chi_\rho$ is the character of $\rho$ and the multiplicities $a_\rho$ do not grow too fast.
 If $D$ is fully elliptic, then the sum in (\ref{at1})  is finite. In general, however,  it may be infinite;   for example, if $X=\mathbf{G}$ is equipped with the left action of $\mathbf{G}$ on itself,  and if $D$ is the zero operator, then $a_\lambda=\dim \chi_\lambda$ and (\ref{at1})
   just becomes the Peter-Weyl decomposition of $L^2(\mathbf{G})$.

\medskip

\noindent A general procedure for computing  the index of 
$D$ is described  in \cite{atgroup}.  Some of the key points are:

\begin{enumerate}[(i)]
\item
The distribution 
 $\ind (D)$ only depends \cite[Theorem (2.6)]{atgroup}
 on the $K$-theory class of the symbol of $D$ in the group $K_\mathbf{G}(T^*_\mathbf{G} X)$,  where $T^*_\mathbf{G} X$ is the subspace of the cotangent bundle $T^*X$ annihilated by the Lie algebra of $\mathbf{G}$ (i.e. ``transverse" to the $\mathbf{G}$-orbits). 
\item
The distribution
$\ind (D)$ is   supported \cite[Theorem (4.6)]{atgroup}
by the conjugacy classes of all elements of $\mathbf{G}$ that have fixed points in $X$.  
\item
Furthermore, 
 $\ind (D)$ is covariant \cite[Theorems (4.1) and (4.3)]{atgroup}
with respect to embeddings $\mathbf{G}\hookrightarrow \mathbf{H}$ and $X\hookrightarrow Z$.
\item \label{A5} 
For a connected Lie group $\mathbf{G}$,
 computation of $\ind (D)$  can be reduced \cite[Theorem (4.2)]{atgroup}
to the case of a maximal torus $\mathbf{T}$, replacing $X$ by $X \times \mathbf{G}/\mathbf{T}$.
\item \label{A6} 
When $\mathbf{G}$ has only finite isotropy groups, the transverse signature operator is transversally elliptic.  Its index equals  the signature \cite[Theorem (10.3)]{atgroup} 
of the rational homology manifold $X/\mathbf{G}$.  
\item
\label{A7}
If, in (\ref{A6}),  $\mathbf{G}$ is a torus,
 then the signature of $X/\mathbf{G}$ can be expressed in cohomological terms involving the (finitely many) fixed point sets $X^{\zap g}$ for ${\zap g}\in \mathbf{T}$.  This formula coincides with the Lefschetz Theorem formula (3.9) of \cite{indexiii}.
\item  
Using (\ref{A5}), the signature formula (\ref{A7}) extends to all $\mathbf{G}$ by using standard results about the flag manifold $\mathbf{G}/\mathbf{T}$.
\end{enumerate}

\noindent In the special case when $\dim X=4$,  the signature formula for an orbifold is just equation (\ref{attwo}), which is what we need in  this paper.  However the general theory also applies to the higher-dimensional case,  when $\dim X=4k$.

\medskip

\noindent It remains to explain how the signature theorem for an (oriented) orbifold $M$ can be derived from this theory of transversally elliptic operators.  In fact, we only need the special case when $M=X/\mathbf{G}$ is the quotient of a manifold by a compact group $\mathbf{G}$ acting with only finite isotropy groups.  The key observation (not widely known, but used already in \cite{kawasaki}) is that the oriented  frame bundle of any oriented 
orbifold is a smooth manifold $P$;   we can thus just take $X=P$ and 
$\mathbf{G}=\mathbf{SO}(n)$.  We just need to to verify that $P$ is a manifold,
and that $\mathbf{SO}(n)$ acts on it with only finite isotropy. 
 This problem is essentially local, we may
  assume that $M=\Gamma\backslash U$, where the finite group 
   $\Gamma \subset \mathbf{SO}(n)$ acts on  $U\approx \RR^n$ by left multiplication;
  moreover, there is  a trivialization 
 $P(U)\cong U\times \mathbf{SO}(n)$ such a manner that
  $\Gamma$ acts freely on $\mathbf{SO}(n)$ on the left, thus commuting with the natural 
  right action by $\mathbf{SO}(n)$.  We now see that  $P(M)= \Gamma \backslash P(U)$
  is therefore a manifold, and that 
  $M= P(M)/ \mathbf{SO}(n)$. Furthermore, 
   every  isotropy group of the 
  $\mathbf{SO}(n)$ action  on $P(M)$ 
   is now a subgroup
 of some isotropy group for the right action of $\mathbf{SO}(n)$ on 
 $\Gamma \backslash\mathbf{SO}(n)$,
 and so is conjugate  to a subgroup of $\Gamma$.

\bigskip

Much of our discussion generalizes nicely to higher dimensions.
The $\mathbf{G}$-signature theorem \cite[(6.12)]{indexiii}
gives an explicit cohomological formula for the signature 
$\tau ({\zap g},X)$, for any ${\zap g}\in \mathbf{G}$, in terms of its fixed point set $X^{\zap g}$. When $Y=X^{\zap g}$ is of codimension $2$ with 
normal cone angle $\beta=2\pi/p$,  $p$ an  integer,  the contribution of $Y$ becomes 
\begin{equation}
\label{at9}
\left\{2^{2k-1} \mathscr{L}(Y)  \coth (\dfrac{y+i\beta}{2})\right\} [Y]
\end{equation}
where $\mathscr{L}$ is the stable characteristic class given by
\begin{equation}
\mathscr{L}=\sum\mathscr{L}_r(p)=\prod \dfrac{x_i/2}{\tanh x_i/2}
\end{equation}
which is essentially the Hirzebruch $L$-series with $x/2$ for $x$.
Here, $y$ denotes the first Chern class of the normal bundle of $Y\subset X$.

\medskip

\noindent The ``defect'' contribution of $Y$ to the signature
$\tau (X)$ is given by replacing $\beta$ by $r\beta$, $1\leq r <p$ in (\ref{at9}), 
 summing over $r$ and dividing by $p$. This gives an explicit polynomial in $y$ and the Pontrjagin classes of $Y$, depending on $\beta$.  In the  $k=1$ case, where $\dim X=4$, we recover the formula (\ref{attwo}).

\medskip

\noindent Finally, one can extend this formula to all $\beta > 0$ , and thereby
 compute the signature defect due to any edge-cone singularity along $Y$.
 To do this,  we proceed as before, first obtaining a formula for
 rational values of $\beta=q/p$, and then extending it to all real values by continuity.

\pagebreak

\vfill 
\noindent 
{\bf Acknowledgements.} The authors would both particularly 
like to express their gratitude 
to the Simons Center for 
Geometry and Physics for its  hospitality    during the Fall semester of 2011, 
during which time this project began to take shape. 
They would also like to thank  fellow Simons Center visitors Nigel Hitchin and 
Sergey Cherkis for many useful discussions of gravitational instantons and related topics. 
The second author would also like to thank Rafe Mazzeo and Xiuxiong Chen for a number of
helpful suggestions.

\end{document}